\def\todaysdate{31\textsuperscript{st} October 2022}
\documentclass[reqno,a4paper]{article}
\usepackage{etex}
\usepackage[T1]{fontenc}
\usepackage[utf8]{inputenc}
\usepackage{lmodern}
\usepackage{ amssymb }
\usepackage{subcaption}
\usepackage{comment}
\usepackage{amssymb,amsmath,amsthm}
\usepackage{thmtools,xcolor}
\usepackage{units}
\usepackage{graphicx}
\usepackage[all]{xy}
\usepackage{tikz}
\usetikzlibrary{arrows,calc,positioning,decorations.pathreplacing}
\usepackage{relsize}
\usepackage{mhsetup}
\usepackage{mathtools}
\usepackage{float}
\usepackage{stmaryrd}
\usepackage{array}
\usepackage{tocloft}
\usepackage{paralist} % For compactitem environment
\usepackage[left=3cm,right=3cm,top=3cm,bottom=3cm]{geometry} % Decrease the margins %
\usepackage[bottom]{footmisc}
\usepackage[colorlinks,citecolor=blue,linkcolor=blue,urlcolor=blue,filecolor=blue,breaklinks]{hyperref}
\usepackage{footnotebackref}
\usepackage[format=plain,indention=1em,labelsep=quad,labelfont={up},textfont={footnotesize},margin=2em]{caption}
\usepackage[mathscr]{euscript}
\usepackage{accents}
\newcolumntype{C}[1]{>{\centering\let\newline\\\arraybackslash\hspace{0pt}}m{#1}}
\definecolor{lightblue}{rgb}{0.8,0.8,1}

\numberwithin{equation}{section}
\numberwithin{figure}{section}

\definecolor{vdarkred}{rgb}{0.7,0,0}
\declaretheoremstyle[
  spaceabove=\topsep,
  spacebelow=\topsep,
  headpunct=,
  numbered=no,
  postheadspace=1ex,
  headfont=\color{vdarkred}\normalfont\bfseries,
  bodyfont=\normalfont\itshape,
]{colored}
\declaretheoremstyle[
  spaceabove=\topsep,
  spacebelow=\topsep,
  headpunct=,
  numbered=no,
  postheadspace=1ex,
  headfont=\normalfont\bfseries,
  bodyfont=\normalfont\itshape,
]{italic}
\declaretheoremstyle[
  spaceabove=\topsep,
  spacebelow=\topsep,
  headpunct=,
  numbered=no,
  postheadspace=1ex,
  headfont=\normalfont\bfseries,
  bodyfont=\normalfont\upshape,
]{upright}

\declaretheorem[style=italic,name=Theorem,numbered=yes,numberwithin=section]{thm}
\declaretheorem[style=italic,name=Lemma,numbered=yes,numberlike=thm]{lem}
\declaretheorem[style=italic,name=Proposition,numbered=yes,numberlike=thm]{prop}
\declaretheorem[style=italic,name=Corollary,numbered=yes,numberlike=thm]{coro}

\declaretheorem[style=upright,name=Definition,numbered=yes,numberlike=thm]{defn}
\declaretheorem[style=upright,name=Remark,numbered=yes,numberlike=thm]{rmk}

\declaretheorem[style=upright,name=Notation,numbered=yes,numberlike=thm]{notation}

\makeatletter
\renewcommand*{\@seccntformat}[1]{\upshape\csname the#1\endcsname.\hspace{1ex}}
\renewcommand*{\section}{\@startsection{section}{1}{\z@}%
	{2.5ex \@plus 1ex \@minus 0.2ex}%
	{1.5ex \@plus 0.2ex}%
	{\normalfont\normalsize\bfseries}}
\renewcommand*{\subsection}{\@startsection{subsection}{2}{\z@}%
	{2.5ex \@plus 1ex \@minus 0.2ex}%
	{-1.5ex \@plus -0.2ex}%
	{\normalfont\normalsize\bfseries}}
\renewcommand*{\subsubsection}{\@startsection{subsubsection}{3}{\z@}%
	{2.5ex \@plus 1ex \@minus 0.2ex}%
	{-1.5ex \@plus -0.2ex}%
	{\normalfont\normalsize\bfseries}}
\renewcommand*{\paragraph}{\@startsection{paragraph}{4}{\z@}%
	{2.5ex \@plus 1ex \@minus 0.2ex}%
	{-1.5ex \@plus -0.2ex}%
	{\normalfont\normalsize\bfseries}}
\renewcommand*{\subparagraph}{\@startsection{subparagraph}{5}{\z@}%
	{2.5ex \@plus 1ex \@minus 0.2ex}%
	{-1.5ex \@plus -0.2ex}%
	{\normalfont\normalsize\slshape}}
\makeatother

\newcommand{\cs}{\mathscr S}
\newcommand{\ct}{\mathscr T}

\newcommand{\C}{C_{3n,n}}
\newcommand{\Cp}{C_{3n-2,n-1}}
\newcommand{\N}{\mathbb N}

\newcommand{\Z}{\mathbb Z}
\newcommand{\LL}{\mathbb L}

\newcommand{\Aut}{\mathrm{Aut}}

\newcommand{\Conf}{\mathrm{Conf}}

\definecolor{dgreen}{RGB}{0,150,0}

\newcommand{\incl}[3][right]%
{%
\draw[<-,>=#1 hook] #2 to ($ #2!0.5!#3 $);
\draw[->,>=stealth'] ($ #2!0.5!#3 $) to #3;%
}
\newcommand{\inclusion}[5][right]%
{%
\draw[<-,>=#1 hook] #4 to ($ #4!0.5!#5 $) node[#2,font=\small]{#3};
\draw[->,>=stealth'] ($ #4!0.5!#5 $) to #5;%
}

{\begin{compactitem}

}%
{\end{compactitem}}
{\begin{compactitem}[#1]

}%
{\end{compactitem}}
{\begin{compactdesc}

}%
{\end{compactdesc}}

\renewcommand{\geq}{\geqslant}
\renewcommand{\leq}{\leqslant}
\renewcommand{\footnoterule}{%
  \kern -3pt
  \hrule width \textwidth height 0.4pt
  \kern 2.6pt
}

\definecolor{dgreen}{RGB}{0,150,0}

\begin{document}
\title{\vspace{-12mm} \Large\bfseries A globalisation of Jones and Alexander polynomials constructed from a graded intersection of two Lagrangians in a configuration space}
\author{ \small Cristina Anghel \quad $/\!\!/$\quad \todaysdate\vspace{-3ex}}
\date{}
\maketitle
{
\makeatletter
\renewcommand*{\BHFN@OldMakefntext}{}
\makeatother
%\footnotetext{2010 \textit{Mathematics Subject Classification}: 20C08, 20C12, 20F36, 55N25, 55R80, 57M10}
%\footnotetext{\textit{Key words and phrases}: Quantum invariants,  Topological models, Graded intersections, Symmetric powers}
}
\vspace{-5mm}
\begin{abstract}
%Jones and Alexander polynomials are two important link invariants which originally came from different areas, and there is still an open question how to see them from the same geometric perspective.
 We consider two Laurent polynomials in two variables associated to a braid, given by {\em graded intersections} between {\em fixed Lagrangians in configuration spaces}.
In order to get link invariants, we notice that we have to quotient by a quadratic relation. Then we prove by topological tools that this relation is sufficient and the first graded intersection gives an invariant which is the Jones polynomial. This shows a {\em topological model for the Jones polynomial} and a direct {\em topological proof}\hspace{0.4mm} that it is a well-defined invariant. The other intersection model in the quotient turns out to be an invariant globalising the Jones and Alexander polynomials. This globalisation in the quotient ring is given by a {\em specific interpolation between the Alexander and Jones polynomials}. 
%Associated to any braid we consider two Laurent polynomials in two variables: a closed intersection model and an open intersection model described by graded intersections between fixed Lagrangians in configuration spaces. Following previous work, these intersections recover the Jones and Alexander polynomials of the closure of the braid through specialisations of coefficients. The aim of this paper is to provide the largest quotient of the Laurent polynomial ring in which these models become invariants. Following an example, we remark that we need to quotient at least by a quadratic relation. Next, we prove directly using homological and topological tools that this relation is sufficient, and the closed graded intersection model gives a well defined link invariant in this quotient. Secondly, we show that this invariant is a multiple of the Jones polynomial. Then, we prove that the open intersection model in the quotient is an invariant which is a {\em specific interpolation between the Alexander and Jones polynomials}, constructed from the {\em grading of the intersection points between two Lagrangians in a fixed configuration space}. As a corollary, we have a {\em topological model for the Jones polynomial} and provide a direct {\em homological proof } that this model gives a well defined invariant (without using skein theory or representation theory).
\end{abstract}
{\tableofcontents}

\section{Introduction}\label{introduction}
Jones and Alexander polynomials are two knot invariants which were defined initially by different tools, but can both be described from skein theory and also representation theory of the quantum group $U_q(sl(2))$ (\cite{RT}, \cite{ADO}). However, they differ from the geometric perspective: the Alexander polynomial is well understood in terms of knot complements but there is an important open problem to describe the Jones polynomial by such means. Further on, categorifications for these two invariants provided by Khovanov homology and Heegaard Floer homology proved to be powerful tools, but which have different natures. It is an important problem to provide geometric categorifications for the Jones polynomial and also to relate such theory to knot Floer homology (\cite{Ras}, \cite{D}, \cite{SM}). Bigelow \cite{Big} provided the first topological model for the Jones polynomial, as a graded intersection of submanifolds in configuration spaces, using the homological representations of braid groups introduced by Lawrence \cite{Law}. They used plat closures of braids and proved the invariance of this model for the Jones polynomial using skein relations. 

In \cite{Cr} we constructed a graded intersection pairing in a configuration space, associated to a braid and taking values in the Laurent polynomial ring in two variables, which recovers the (coloured) Jones polynomial and (coloured) Alexander polynomial through specialisations of coefficients to polynomials in one variable. Based on this result, we pose the following question: what is the largest ring in which this topological model provides link invariants? In this paper we show that it is necessary to quotient by a quadratic relation and in this case this construction provides a {\em topological model} for an {\em interpolation between Jones and Alexander polynomials}, constructed in a quotient of the Laurent polynomial ring by a quadratic relation. We work with links seen as braid closures, as opposed to plat closures, of braids.
\vspace{-2mm}
\subsection{Main result} For $n,m \in \N$, we define $C_{n,m}$ to be the unordered configuration space of $m$ points in the $n$-punctured disc $\mathscr D_n$. We will construct two graded intersections in such configuration spaces in the punctured disc: $\Omega(\beta_n), \Omega'(\beta_n) \in \Z[x^{\pm1},d^{\pm 1}]$ for $\beta_n \in B_n$, which will be parametrised by the intersection points between two fixed Lagrangian submanifolds, graded in a certain way. The construction of the Lagrangians is done by fixing a collection of arcs/ circles in the punctured disc, taking their product and considering its image in the quotient to the unordered configuration space (figure \ref{IntroL}). 

In order to answer the problem coming from \cite{Cr}, we compute an example (Section \ref{S:3}) and remark that in order to obtain invariants from these topological models we should quotient the Laurent polynomial ring by a quadratic relation and work in this quotient (denoted by $\LL$). Further on we proceed as follows:
\begin{itemize}
\item[•] We prove by {topological and homological techniques} that the intersection form $\Omega(\beta_n)$ becomes {invariant under the Markov moves in this quotient $\LL$}, so it gives a {well defined link invariant}.
\item[•] Then, we compute these two intersection forms $\Omega'(\beta_n)$ and $\Omega(\beta_n)$ in this quotient.
\end{itemize}
The main results that we obtain are the following:
\begin{itemize}
\item[•] The open intersection form $\Omega'$ becomes an {\em interpolation between Jones and Alexander polynomials} given directly by a {\em graded intersection of two Lagrangians in a configuration space}, over the quotient ring $\LL$.
\item[•] We provide an {\em intrinsic homological construction of the Jones polynomial} and a {purely homological proof that it is a well-defined link invariant} (using the intersection $\Omega$).
\item[•] Also, we obtain a {\em general method for checking invariance under the Markov moves} of constructions based on {\em Lawrence type representations}.
\end{itemize}
\vspace{-4mm}
\subsection{Description of the models}
For the first intersection model, $\Omega(\beta_n)$, we start with $\cs$ and $\ct$ which are the Lagrangian submanifolds given by the collections of red arcs and green circles from the left hand side of figure \ref{IntroL}, in the configuration space of $n$ points in the $(3n)$-punctured disc. The second intersection pairing, $\Omega'(\beta_n)$, is constructed using the Lagrangian submanifolds encoded by the collections of red arcs and green circles from the right hand side of figure \ref{IntroL}, $\cs'$ and $\ct'$, seen in the configuration space of $n-1$ points in the $(3n-2)$-punctured disc.
\begin{figure}[H]
$$ \ \ \ \ \ \ \ \ \ \ \ \ \  \cs,\ct \subseteq \text{Conf}_{n}(\mathscr D_{3n})  \ \ \ \ \ \ \ \ \ \ \ \ \ \ \ \ \ \ \ \ \ \ \ \ \   \ \ \ \ \ \ \ \ \ \cs',\ct'\subseteq \text{Conf}_{n-1}(\mathscr D_{3n-2}) \ \ \ \ \ \ \ \ \ \ $$
\vspace{-5mm}
\centering
\includegraphics[scale=0.305]{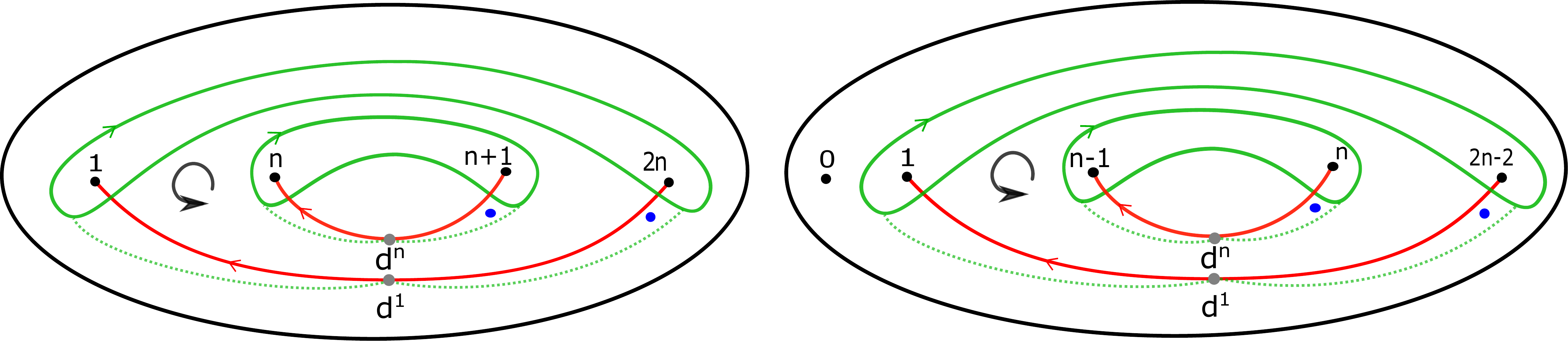}
\vspace{-3mm}
\hspace{-10mm}\caption{Closed intersection $\Omega(\beta_n)$ \hspace{35mm} Open intersection $\Omega'(\beta_n)$ \ \ \ \ \ \ \ \ \ \ \ \ \ \ \ \  }\label{IntroL}
\end{figure}

\vspace{-4mm}
We denote by $\mathbb I_{m}$ the trivial braid with $m$ strands. For the next part, we see the braid groups $B_{3n}$ and $B_{3n-2}$ as the mapping class groups of the $(3n)-$punctured disc and $(3n-2)-$punctured disc respectively. This will lead to two well-defined Lagrangians: $$(\beta_n\cup \mathbb I_{2n}) \ \cs  \subseteq C_{3n,n}; \ \ \ (\beta_n\cup \mathbb I_{2n-2}) \ \cs' \subseteq C_{3n-2,n-1}$$
which are associated to a braid $\beta_n\in B_n$. We consider the sets of intersection points:
\begin{equation}
I_{\beta_n}=(\beta_n \cup \mathbb I_{2n}) \cs\cap \ct; \ \ \ I'_{\beta_n}=(\beta_n \cup \mathbb I_{2n-2}) \cs'\cap \ct'.
\end{equation}
Then, we present two graded intersections, denoted by $\langle  (\beta_n \cup \mathbb I_{2n}) \cs, \ct \rangle$ and $\langle  (\beta_n \cup \mathbb I_{2n-1}) \cs', \ct' \rangle$, which are parametrised by the set of intersection points between the above Lagrangians and graded using a local system, as presented in relation \eqref{int}.

The blue punctures from Figure \ref{IntroL} play an important role in the grading procedure, and from the algebraic perspective they correspond to the quantum trace which is associated to the representation theory of the quantum group $U_q(sl(2))$.
\begin{defn}(Graded intersections)\label{defn} Let us consider the following polynomials:
$$\Omega(\beta_n)(x,d), \Omega'(\beta_n)(x,d) \in \Z[x^{\pm \frac{1}{2}}, d^{\pm 1}],$$
which are defined from graded intersections using the Lagrangian submanifolds from Figure \ref{IntroL}:
\begin{equation}
\begin{aligned}
& \Omega(\beta_n)(x,d):=(d^2x)^{\frac{w(\beta_n)+n}{2}} \cdot d^{-n}\langle  (\beta_n \cup \mathbb I_{2n}) \cs, \ct \rangle\\
& \Omega'(\beta_n)(x,d):=(d^2x)^{\frac{w(\beta_n)+n-1}{2}} \cdot d^{-(n-1)}\langle  (\beta_n \cup \mathbb I_{2n-1}) \cs', \ct' \rangle.
\end{aligned}
\end{equation}
Here, $w(\beta_n)$ is the writhe of the braid $\beta_n$. We call $\Omega(\beta_n)(x,d)$ the graded intersection associated to the closed model and $\Omega'(\beta_n)(x,d)$ the graded intersection corresponding to the open model.
\end{defn}
Let $\tilde{J}(L)$ be the normalised Jones polynomial and $J(L)$ the un-normalised Jones polynomial (Notation \ref{Jones}). In \cite{Cr}, we have proved that the open intersection model recovers the Jones and Alexander polynomials of the closure of the braid, through the following specialisations of coefficients:
\begin{equation}\label{eq:1}
\begin{aligned}
&\Omega'(\beta_n)(x,d)|_{x=d^{-1}}=\tilde{J}(\hat{\beta}_n,x)\\
&\Omega'(\beta_n)|_{d=-1}=\Delta(\hat{\beta}_n,x).
\end{aligned}
\end{equation}
\vspace{-7mm}
\subsection{Invariants in the quotient ring}
Here we start with the problem concerning the invariance of the closed intersection form $\Omega$. Let $\LL:=\Z[x^{\pm \frac{1}{2}}, d^{\pm 1}]/\left( (d+1)(dx-1)\right)$ and consider the quotient morphism:
$$\overline{\phantom{a}} : \Z[x^{\pm \frac{1}{2}}, d^{\pm 1}]\rightarrow \Z[x^{\pm \frac{1}{2}}, d^{\pm 1}]/\left( (d+1)(dx-1)\right).$$
\vspace{-5mm}
\begin{thm}[Invariant in the quotient ring]\label{THEOREM}
The ring $\LL$ is the largest quotient of $\Z[x^{\pm \frac{1}{2}}, d^{\pm 1}]$ such that the image of the intersection form $\Omega(\beta_n)$ in this quotient becomes a link invariant. More precisely, let us denote the image of the graded intersection in this quotient ring by:
\begin{equation}
 \bar{\Omega}(\beta_n)(x,d) \in \Z[x^{\pm\frac{1}{2}}, d^{\pm 1}]/\left((d+1)(dx-1) \right)
\end{equation}
Then $\bar{\Omega}(L)(x,d):=\bar{\Omega}(\beta_n)(x,d)$ is a well defined link invariant for an oriented link $L$ which is the closure of $\beta_n$. 
Also, if $\LL'$ is a quotient of the Laurent polynomial ring in which $\Omega(\beta_n)$ becomes a link invariant then the quotient onto $\LL'$ factors through $\LL$.
\end{thm}
The proof of this result is topological. We use homological tools in order to prove that  $\bar{\Omega}(\beta_n)$ is invariant under the two Markov moves. 
\vspace{-2mm}
\subsection{The intersection forms recover the Jones and Alexander polynomials}
In the next part we want to understand this invariant, which takes values in the quotient of the Laurent polynomial ring. First, we prove that it has two specialisations, one of which recoveres the Jones polynomial and the other one which vanishes. 
\begin{thm}\label{Tsk}
 The closed intersection form $\bar{\Omega}(L)$ specialises to the un-normalised Jones polynomial and vanishes for the specialisation of coefficients associated to roots of unity:
\begin{equation}\label{eqC:3}
\begin{aligned}
&\bar{\Omega}(L)|_{x=q^{2}, d=q^{-2}}=J(L,q)\\
&\bar{\Omega}(L)|_{d=-1}=0. 
\end{aligned}
\end{equation}
\end{thm}
The proof of this models is also topological. Using homological representations, we show that it is enough to verify two particular skein type relations and we check this by intersecting curves in the punctured disc. 

Secondly, we also prove that the two specialisations of the open intersection model $\bar{\Omega}'$ satisfy the the skein relations characterising the Jones and Alexander polynomials. 
\begin{thm}\label{Tsk'}
 The open intersection form $\bar{\Omega}'(\beta_n)$ specialises to the normalised Jones polynomial and the Alexander polynomial of the closure of the braid, as below:
\begin{equation}\label{eqC:3'}
\begin{aligned}
&\bar{\Omega}'(\beta_n)|_{x=q^{2};d=q^{-2}}=\tilde{J}(L,q)\\
&\bar{\Omega}'(\beta_n)|_{d=-1}=\Delta(L,x). 
\end{aligned}
\end{equation}
\end{thm}
For this, we start from the fact that these specialisations are conjugacy invariants, which comes from \cite{Cr}. Then we prove directly by the same topological techniques as the ones used for $\bar{\Omega}$ that they satisfy the skein relations. This makes this paper self-contained and independent of all identifications from \cite{Cr} (presented in equation \eqref{eq:1}) except the assumption of the conjugacy invariance of the specialisations of the open model $\Omega'$.
\vspace{-4mm}
\subsection{The explicit form of the invariants in the quotient ring} Now we want to understand what the two intersection models in the quotient ring are. We use algebraic arguments and, starting from the fact that the two models recover the Jones and Alexander polynomials, we conclude that in the quotient ring they have to be interpolations between these two invariants, as below.
\begin{thm}[The closed intersection model as the Jones polynomial]\label{THEOREM3} The closed model has the following form in the quotient ring $\LL$:
\begin{equation}
\bar{\Omega}(L)(x,d)=x^{\frac{1}{2}}(d+1)\cdot \tilde{J}(L)(x).
\end{equation}
\end{thm}
\begin{thm}[The open intersection model interpolates between the Jones and Alexander polynomials]\label{THEOREM2}
The open model is a well defined invariant in the quotient, which is given by:
\begin{equation}
\bar{\Omega}'(L)(x,d)=\Delta(L)(x)+(d+1) \cdot \frac{\tilde{J}(L)(x)-\Delta(L)(x)}{(x^{-1}+1)} \text{  in  } \LL.
\end{equation}
\end{thm}
\vspace{-4mm}
\subsection{Further work} 

\

{\bf Categorifications} One motivation for this research direction concerns the description of geometrical type categorifications for Jones and Alexander polynomials and relations between them (such as spectral sequences \cite{Ras}, \cite{D}, \cite{SM}). We expect that there is a categorification procedure from the specialised open model $\bar{\Omega}'(L)_{d=-1}$ which gives knot Floer homology (for knots, the geometric supports of the Lagrangians from our picture are Heegaard diagrams). We are interested in studying this machinery directly at the interpolation level $\bar{\Omega}'(L)$ (over $\LL$), where we have a grading for the intersection points given by two variables. In particular, we are interested in investigating the grading for this interpolation model, which is geometrically described as in section \ref{S:1}, and relating this to the grading from the Floer homology picture and certain gradings for possible geometrical categorifications for the Jones polynomial. 

{\bf Twisted invariants from Lawrence representations} This work is also part of a wider joint project with Fathi Ben Aribi (\cite{Fathi}) where we aim to define twisted invariants for knots starting from twisted Lawrence type representations. In section \ref{S:3} we see that the intersection form $\bar{\Omega}$ is given also by a {\em sum of traces of Lawrence representations} and we prove that this is in turn {\em invariant under Markov moves}. This provids a {\em method for checking Markov moves on constructions defined using Lawrence type representations}. This method is a starting point in this joint work, where we aim to check Markov moves for twisted versions of Lawrence representations.  
\vspace{-3mm}
\subsection*{Structure of the paper}
In Section \ref{S:1} we present the grading procedure and the construction of the two intersection models in the configuration space. In Section \ref{S:3}, we compute the closed model $\Omega$ for the unknot and stabilised unknot and deduce that in order to have invariance we need a quadratic relation. Then, in Section \ref{S:2} we discuss the relation between these intersection models and two state sums of Lagrangian intersections, defined using pairings between Lawrence representations, from \cite{Cr2}. Section \ref{S:4} is devoted to the proof of the invariance of the closed intersection form $\bar{\Omega}$ under the Markov moves (in the quotient ring). After that, in Section \ref{S:7}, we prove topologically that the specialisations of the two intersection forms $\bar{\Omega}$ and $\bar{\Omega}'$ satisfy appropriate skein relations and so they are the Jones and Alexander polynomials. In Section \ref{S:5} we show that these two models become interpolations of Jones and Alexander polynomials in the quotient. In the last section we compute the open intersection model for the trefoil knot and check that it is given by the above interpolation. 
\vspace*{-3mm}
\subsection{Acknowledgements} I would like to thank Rinat Kashaev very much for discussions concerning the form of these intersection models in the quotient ring and the conclusion that they are the above interpolations between Jones and Alexander polynomials. I would also like to thank Emmanuel Wagner for useful discussions regarding the first version of the paper. I acknowledge the support of SwissMAP, a National Centre of Competence in Research funded by the Swiss National Science Foundation.

\section{Notations} We start with the quotient ring:
\begin{equation}
\LL=\Z[x^{\pm \frac{1}{2}}, d^{\pm 1}]/\left( (d+1)(dx-1)\right).
\end{equation}
\begin{rmk}
Looking at $\LL$ as an algebra over $\Z[x^{\pm\frac{1}{2}}]$, we have a basis given by: $$\{d^{i} \mid 0\leq i\leq 1\}.$$ In other words, the powers of $d$ which are bigger than $2$ can be expressed in terms of the above basis. For example, we have:
\begin{equation}\label{relations}
\begin{aligned}
&d^{2}=x^{-1}d-d+x^{-1}\\
&d^3=(x^{-2}-x^{-1}+1)d+x^2-x^{-1}.
\end{aligned}
\end{equation}
\end{rmk}

\begin{notation}\label{ind}
For a set of indices $\bar{i}=(i_1,...,i_n)$ where $i_1,...,i_n\in \{0,1\}$ we denote the symmetric set of indices by:
\begin{equation}
1-\bar{i}:=(1-i_n,...,1-i_1).
\end{equation} 
Also, we will change the coefficients for certain modules, using the following definition.
\begin{notation}\label{N:spec} 
Let $R$ be a ring and consider $M$ an $R$-module which has a basis $\mathscr B$.
We consider $S$ to be another ring and let us suppose that we have a specialisation of coefficients, given by a morphism:
$$\psi: R \rightarrow S.$$
The specialisation of the module $M$ by the morphism $\psi$ is the following $S$-module: $$M|_{\psi}:=M \otimes_{R} S.$$ It will have a basis described by:
$$\mathscr B_{M|_{\psi}}:=\mathscr B \otimes_{R}1 \in M|_{\psi}. $$
\end{notation}

\end{notation}
\begin{defn}(Specialisations of coefficients towards one variable)\label{N}
We consider two specialisations of coefficients, given by:
\begin{equation}
\begin{aligned}
\psi_{J}: \Z[x^{\pm \frac{1}{2}}, d^{\pm 1}]/& \left((d+1)(dx-1) \right) \rightarrow \Z[x^{\pm \frac{1}{2}}]\\
&\psi_{J}(d)=x^{-1}; \psi_{J}(x)=x.
\end{aligned}
\end{equation}
\begin{equation}
\begin{aligned}
\psi_{\Delta}: \Z[x^{\pm \frac{1}{2}}, d^{\pm 1}]/&\left((d+1)(dx-1) \right)\rightarrow \Z[x^{\pm \frac{1}{2}}]\\
&\psi_{\Delta}(d)=-1; \psi_{\Delta}(x)=x.
\end{aligned}
\end{equation}
\end{defn}
We will use these two changes of coefficients $\psi_{J}$ and $\psi_{\Delta}$ in order to pass from the intersection form from the quotient ring in two variables $\LL$ towards the Jones polynomial and Alexander polynomial respectively. 

\begin{figure}[H]
\begin{center}
\begin{tikzpicture}
[x=1.2mm,y=1.4mm]

% Nodes of the diagram
\node (b1)  at   (44,0)  {$\Z[x^{\pm \frac{1}{2}}]$};
\node (b5)  at  (0,0) {$\Z[x^{\pm \frac{1}{2}}]$};
\node (b4) at (22,15)   {$\Z[x^{\pm \frac{1}{2}}, d^{\pm 1}]/\left((d+1)(dx-1) \right)$};
\node (s1)  at   (0,6)  {$d=x^{-1}$};
\node (s2)  at   (44,6)  {$d=-1$};
\draw[<-]  (b1)      to node [left,yshift=-2mm,font=\large]{$\psi_{\Delta}$}   (b4);
\draw[<-]   (b5) to node [right,yshift=-2mm,font=\large] {$\psi_{J}$}                        (b4);
\end{tikzpicture}
\end{center}
\end{figure}
\begin{defn} (Specialisations for Lawrence representations) Let us denote the quotient morphism to $\LL$ by:
\begin{equation}\label{NN}
s:\Z[x^{\pm1},d^{\pm1}]\rightarrow\Z[x^{\pm1},d^{\pm1}]/((d-1)(xd-1)).
\end{equation}
\end{defn}
\begin{notation}\label{Jones}
We consider $\tilde{J}(L)$ to be the normalised Jones polynomial, whose evaluation on the unknot $\mathscr U$ is the following:
$$\tilde{J}(\mathscr U)(x)=1.$$ 
Also, let $J(L)$ be the un-normalised version of the Jones polynomial, which is given by:
$$J(\mathscr U)(x)=x^{\frac{1}{2}}+x^{-\frac{1}{2}}.$$ 
\end{notation}
\section{Definition of the intersection forms}\label{S:1} In this part we introduce the two intersection forms, which are given by certain graded intersections in the configuration spaces in the punctured disc.

For $n,m\in \N$ we consider the unordered configuration space of $m$ points in the $n$-punctured disc and denote it by
$C_{n,m}:=\Conf_{m}(\mathscr D_{n}).$
We consider a collection of base points $d_1,...,d_m \in \mathscr D_n$ and the associated base point in the configuration space ${\bf d}=\{d_1,...,d_m\}\in C_{n,m}$.

\vspace{-1mm}
\begin{center}
\begin{tikzpicture}[scale=0.95]
\foreach \x/\y in {-0.3/2,2/2,4/2,2/1,2.5/1,3/1.06} {\node at (\x,\y) [circle,fill,inner sep=1pt] {};}
\node at (-0.1,2.6) [anchor=north east] {$1$};
\node at (2.2,2.6) [anchor=north east] {$i$};
\node at (4.2,2.6) [anchor=north east] {$n$};
\node at (3,2) [anchor=north east] {$\sigma_i$};
\node at (2.2,1) [anchor=north east] {$\tiny \mathrm d_1$};
\node at (2.8,1) [anchor=north east] {$\tiny \mathrm d_2$};
\node at (3.6,1) [anchor=north east] {$\tiny \mathrm d_m$};
\node at (2.63,2.3) [anchor=north east] {$\wedge$};
\draw (2,1.8) ellipse (0.4cm and 0.8cm);
\draw (1.9,1.8) ellipse (3.65cm and 1.65cm);
\foreach \x/\y in {7/2,9/2,11/2,8.9/1,9.5/1,10/1.07} {\node at (\x,\y) [circle,fill,inner sep=1pt] {};}
\node at (7.2,2.6) [anchor=north east] {$1$};
\node at (9.2,2.6) [anchor=north east] {$i$};
\node at (11.2,2.6) [anchor=north east] {$n$};
\node at (9.2,1) [anchor=north east] {$\tiny \mathrm d_1$};
\node at (9.8,1) [anchor=north east] {$\tiny \mathrm d_2$};
\node at (10.6,1) [anchor=north east] {$\tiny \mathrm d_m$};
\node at (9,1.5) [anchor=north east] {$\delta$};
\draw (9.4,1.8) ellipse (3.65cm and 1.65cm);
%\draw[->, color=blue, very thick]             (12.2,1)     to[in=30,out=30] node[yshift=-3mm,font=\small]{}  (12.5,1.02);
\draw (9.5,1)  arc[radius = 3mm, start angle= 0, end angle= 180];
\draw [->](9.5,1)  arc[radius = 3mm, start angle= 0, end angle= 90];
\draw (8.9,1) to[out=50,in=120] (9.5,1);
\draw [->](8.9,1) to[out=50,in=160] (9.25,1.12);
\end{tikzpicture}
\end{center}
For the grading procedure, we will use a local system on this configuration space.
\begin{notation} For this, we start with the abelianisation map $ab: \pi_1(C_{n,m}) \rightarrow H_1(C_{n,m})$. Then, for any $m\geq 2$ we have: 
\begin{equation*}
\begin{aligned}
H_1(C_{n,m}) \ & \simeq \ \ \ \ \Z^{n} \ \ \oplus \ \ \Z \\
& \hspace{6mm} \langle ab(\Sigma_i)\rangle \ \ \langle ab(\Delta)\rangle, \  {i\in \{1,...,n\}}.
\end{aligned}
\end{equation*}
More specifically, the generators of the first group, $\Z^{n}$, are given by classes of loops defined by the property that their first component goes around the $i$-th puncture and the other components are constant:
$\Sigma_i(t):=\{ \left(\sigma_i(t), d_2,...,d_m \right) \}, t \in [0,1].$
The generator of the second group is the class of a loop $\Delta$ which swaps the first two base points: $\Delta(t):=\{ \left(\delta(t), d_3,...,d_m \right) \}, t \in [0,1].$
\end{notation}

For the next part we fix $l \in \N$ such that $l \leq n$. We will next use a morphism that distinguishes the $n$ punctures, by separating them into two sets: $n-l$ black punctures and $l$ blue punctures, as in figure \ref{Localsystem}. Further on, for orientation purposes, we also fix a number $k\in \{0,...,n-l\}$.
\begin{center}
\begin{figure}[H]
\centering
\includegraphics[scale=0.27]{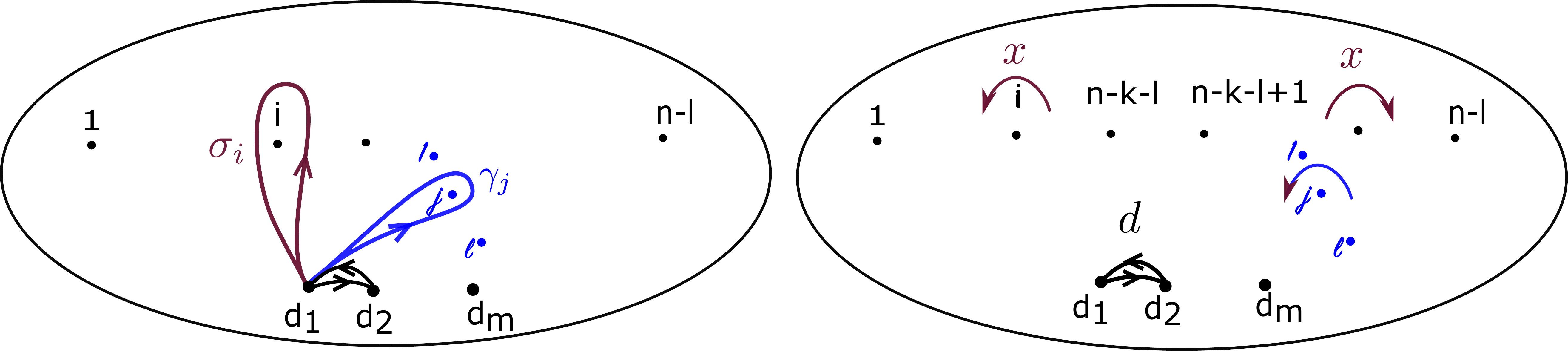}
\hspace{-10mm}\caption{Local system }\label{Localsystem}
\end{figure}
\end{center}
\vspace{-8mm}
\begin{defn}(Local system)\label{localsystem} For $l \in \{1,...,n\}$ and $k \in \{0,...,n-l\}$, we define the following morphism: 
\begin{equation}
\begin{aligned}
&  \ \hspace{-6mm} \text{ab} \ \ \ \ \ \ \ \ \ \ \ \ \ \ \ \ \ \ \ \ \ \ \ \ f\\
 \phi: \pi_1\left(C_{n,m} \right) \ \rightarrow \ & \ \Z^{n-l} \oplus \Z^{l} \oplus \Z  \ \rightarrow  \ \Z \ \oplus \Z \oplus \Z\\
& \langle [\sigma_i] \rangle \ \ \langle [\gamma_j]\rangle \ \ \langle [\delta]\rangle \ \ \ \ \langle x \rangle \ \  \langle y \rangle \ \ \langle d \rangle\\
&{i\in \{1,...,n-l\}}, j\in \{1,...,l\}\\
&\ \hspace{-28mm} \phi=f \circ ab.
\end{aligned}
\end{equation}
Here, the morphism $f$ is defined as augmentations on the first two summands as follows:
\begin{equation}
\begin{cases}
&f(\sigma_i)=x, i\in \{1,...,n-k-l\}\\
&f(\sigma_i)=-x, i\in \{n-k-l+1,...,n-l\}\\
&f(\gamma_j)=y, j\in \{1,...,l\}\\
&f(\delta)=d.
\end{cases}
\end{equation}
\end{defn}
\subsection{The open and closed intersection forms}
In the next part we present the grading procedures that will be used for the definition of the two graded intersections. For $\Omega$, we will work in the setting from Definition \ref{localsystem} where the ambient space and the parameters $n,l,k$ are: $$\Big(\C=\Conf_{n}(\mathscr D_{3n}), n\rightarrow 3n, l\rightarrow n, k\rightarrow n \Big).$$ For $\Omega'$ we will use the following data:
$$\Big(\Cp=\Conf_{n-1}(\mathscr D_{3n-2}), n\rightarrow 3n-2, l\rightarrow n-1, k\rightarrow n-1 \Big).$$
The construction of the graded intersection $\Omega'$ is precisely the one presented in \cite{Cr}, for the case where the colour $N=2$. 
This is parametrised by the set of intersection points between the two Lagrangians from figure \ref{IntroL}:
\begin{equation}
I'_{\beta_n}:=(\beta_n \cup \mathbb I_{2n-1}) \cs'\cap \ct'
\end{equation}
together with a grading coming from a certain local system defined on the configuration space.

For the closed model $\Omega$, the intersection will be defined in an analog manner, where we add one particle in our configuration space, as below.
\subsubsection{Grading for $\Omega$}
   Let us fix $\beta_n \in B_n$. Using the property that the braid group is the mapping class group of the punctured disc, we act with such a braid (to which we add $2n$ trivial strands) on $\cs$ and consider the submanifold:
$$(\beta_n \cup \mathbb I_{2n}) \cs \subseteq \C.$$
We choose a representative of this action such that it is supported in a grey disk around the punctures labeled by $\{1,...,n\}$, and also such that the above submanifold is a Lagrangian submanifold, as discussed in \cite{Cr} Section 2.2. 

Further on, we define the graded intersection pairing, which is generated by the set of intersection points, denoted by:
\begin{equation}
I_{\beta_n}:=(\beta_n \cup \mathbb I_{2n}) \cs\cap \ct
\end{equation}
and graded by the above local system.
The grading procedure will be done by associating to each intersection point $\bar{x}$ a loop $l_{\bar{x}}$ in the configuration space, which will be evaluated by the morphism $\phi$:
$$x \in I_{\beta_n} \ \ \rightsquigarrow \ \ l_{\bar{x}} \ \ \rightsquigarrow \ \  \phi(l_{\bar{x}}).$$
In order to prescribe the loop, we use a base point which is chosen on the submanifold $\cs$, and we denote it as ${\bf d}:=(d^1,...,d^{n})$, using picture \ref{Diffeo}.
Let us denote by
$s\cs, s\ct \subseteq \mathscr D_{3n}$ the collections of $n$ red curves and $n$ green circles which give the geometric supports for $\cs$ and $\ct$.
\begin{defn}(Paths to the base points) For the next part, we fix a set of $n$ paths in the punctured disc  which connect the red curves with the left hand side of the circles, as in figure \ref{Diffeo}, and denoted them by $\eta_1,...,\eta_{n}$. Similarly, we consider a collection of paths $\eta'_1,...,\eta'_{n}$ which start from the red curves and end on the right hand side of the circles.\end{defn}
\vspace{-3mm}
\begin{figure}[H]
\centering
\includegraphics[scale=0.4]{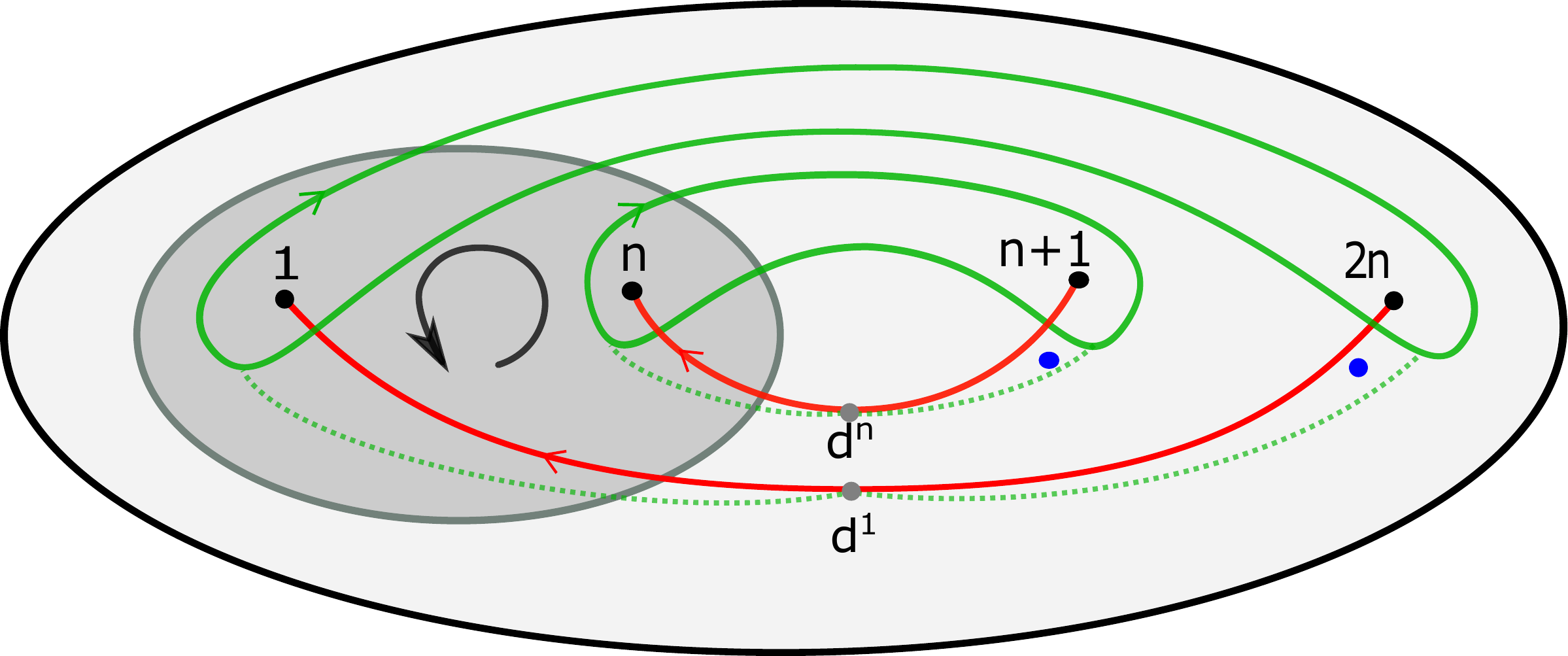}
\vspace{-1mm}
\caption{Braid action}
\label{Diffeo}
\end{figure}

\begin{defn}(Loop associated to an intersection point)
Let $\bar{x}=(x_1,...,x_{n}) \in I_{\beta_n}$. The loop, based in $\bf{d}$, will be constructed in two steps.

For the first part, we fix $k \in \{1,...,n\}$ and use the $k^{th}$ green circle which goes around the punctures $(k, 2n+1-k)$. There is exactly one component of $\bar{x}$, denoted by $x_{\iota(k)}$, which belongs to this circle.

If $x_{\iota(k)}$ is in the left hand side of the punctured disc, we define $\bar{\nu}^k$ to be the path (in the punctured disc) which starts from $d^k$ following $\eta_k$ and then continue on the green curve until it reaches the point $x_{\iota(k)}$, as in figure \ref{Picture1}. If $x_{\iota(k)}$ is on the right hand side of the punctured disc, then we do a similar procedure and define a path $\bar{\nu}^k$ by using the path $\eta'_k$ to begin with, and then go to the intersection point following part of the green circle.

\begin{figure}[H]
\centering
\includegraphics[scale=0.4]{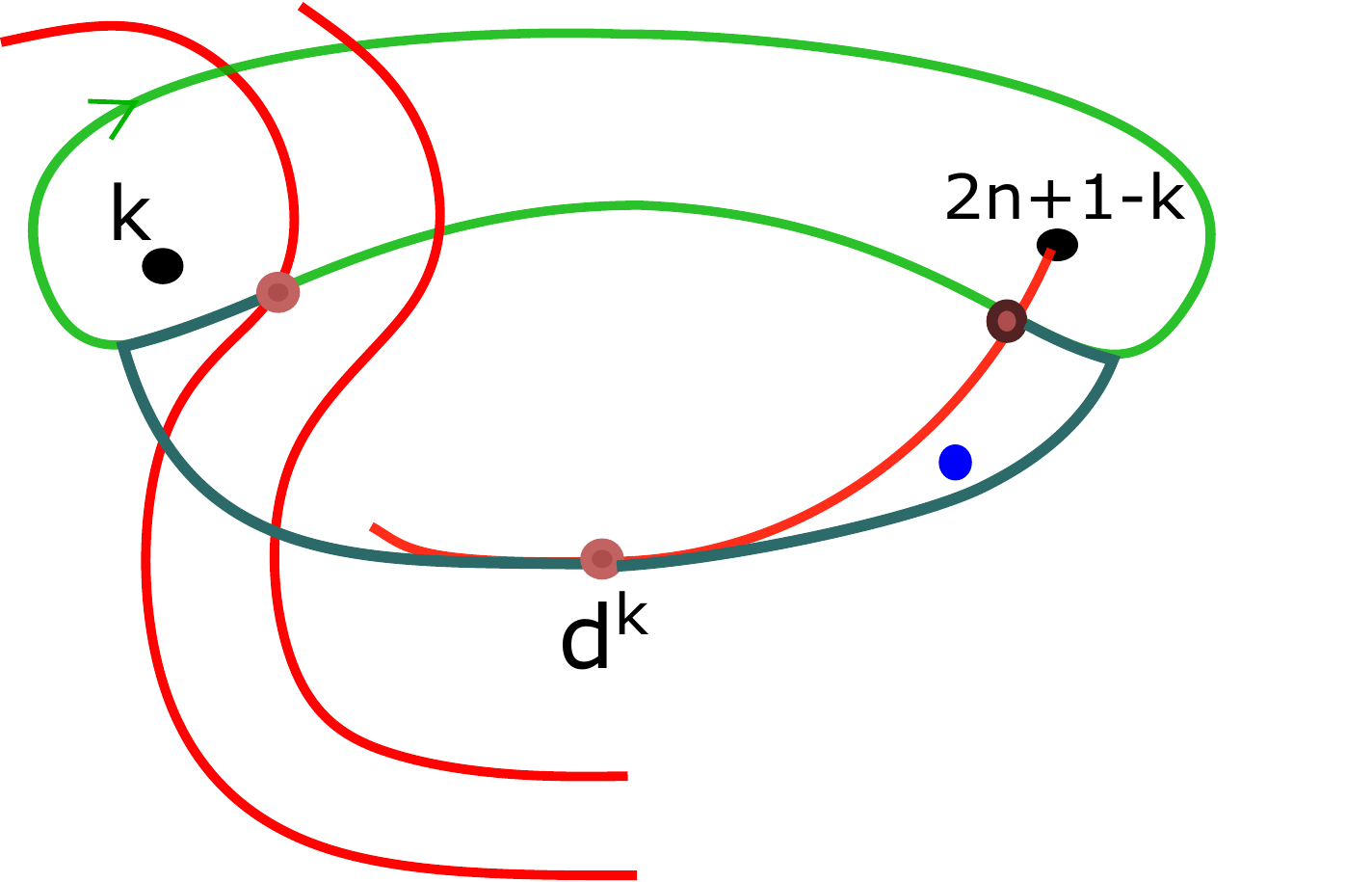}
\caption{Paths from the base points}\label{Picture1}
\end{figure}
\vspace{-33mm}
$$\eta_k \ \ \ \ \ \ \ \ \ \ \ \ \ \ \ \ \ \ \ \ \ \ \ \ \ \ \ \ \ \ \ \ \ \ \eta'_k \ \ \ \ \ \ \ \ \ \ $$

\vspace{20mm}

Doing this construction for all $k\in \{1,...,n\}$ we get a collection of $n$ paths in the punctured disc. Now we consider the path in the configuration space from $\bf d$ to $\bar{x}$ given by the union of these paths:
\begin{equation}
\bar{\gamma}_{\bar{x}}:= \{\bar{\nu}^1,...,\bar{\nu}^{n}\}. 
\end{equation}

For the second part of the construction, we start from the components of $\bar{x}$ and we go back to the base points in the punctures disc using the red arcs. More precisely, each component $x_k$ of $\bar{x}$ belongs to a unique red arc, denoted by $j(k)$. Let $\nu_k$ be the path in the punctured disc starting in $x_k$ and ending in $d^{j(k)}$ following the $j(k)^{th}$ red curve. Now, we look at the path in the configuration space from $\bar{x}$ to $\bf{d}$ given by this collection of paths, and denote it as below:
\begin{equation}
\gamma_{\bar{x}}:= \{\nu^1,...,\nu^{n}\}. 
\end{equation}
Now we define the loop associated to the intersection point $\bar{x}$ (base in $\bf d$) as the composition of the two previous paths:
\begin{equation}
l_{\bar{x}}:=\gamma_{\bar{x}} \circ \bar{\gamma}_{\bar{x}}. 
\end{equation}
\end{defn}
\subsection{Graded intersections} In this part we recall the definition of the graded intersection defined in \cite{Cr}. After that, we consider a smaller local system and use that to define the graded intersection which we use for the main result. 
\begin{defn}\label{D:int0} We consider a graded intersection $\ll (\beta_n\cup \mathbb I_{2n}) \cs,\ct \gg \ \in \Z[x^{\pm1},y^{\pm 1},d^{\pm 1}]$, which is parametrised by the intersection points and graded using the associated loops and the local system as below:
\begin{equation}\label{int0}
\ll (\beta_n\cup \mathbb I_{2n}) \cs,\ct \gg:= \sum_{\bar{x}\in I_{\beta_n}} \epsilon_{x_1}\cdot...\cdot \epsilon_{x_n}\cdot \phi(l_{\bar{x}}).
\end{equation}
In this formula $\epsilon_{x_i}$ is the sign of the local intersection between the circle and the red curve that $x_i$ belongs to, in the punctured disc.
\end{defn}
For the grading procedure that we need for this model, we will use a further quotient which is defined as follows:
\begin{equation}
\begin{aligned}
F:~ & \Z[x^{\pm1},y^{\pm1},d^{\pm1}] \rightarrow \Z[x^{\pm1},d^{\pm1}]\\
& \hspace{-5mm}\begin{cases}
F(x)=x; \\
F(y)=-d; F(d)=d.
\end{cases}
\end{aligned}
\end{equation}
\begin{defn}(Change of coefficients) Let us define the morphism $\tilde{\phi}$ which is obtained from $\phi$ and has the co-domain the group ring $\Z[\Z \oplus \Z \oplus \Z]\simeq \Z[x^{\pm1},y^{\pm1},d^{\pm1}]$:
$$\tilde{\phi}:\pi_1(C_{n,m})\rightarrow \Z[x^{\pm1},y^{\pm1},d^{\pm1}].$$ 
Then, we define the morphism obtained from $\tilde{\phi}$ by taking the quotient using $F$:
\begin{equation}
\begin{aligned}
&\varphi:\pi_1(C_{n,m})\rightarrow \Z[x^{\pm1},d^{\pm1}]\\
&\varphi=F \circ \tilde{\phi}.
\end{aligned}
\end{equation}
\end{defn}
\begin{defn}(Grading)\label{D:int} Let us define the following graded intersection:
\begin{equation}\label{int}
\begin{aligned} 
& \langle (\beta_n\cup \mathbb I_{2n}) \cs,\ct \rangle  \in \Z[x^{\pm1},d^{\pm 1}]\\
&\langle (\beta_n\cup \mathbb I_{2n}) \cs,\ct \rangle:= \sum_{\bar{x}\in I_{\beta_n}} \epsilon_{x_1}\cdot...\cdot \epsilon_{x_{n}}\cdot \varphi(l_{\bar{x}}).
\end{aligned}
\end{equation}
\end{defn}
We remark that:
\begin{equation*}
\begin{aligned}\label{eq:2}
\langle (\beta_n\cup \mathbb I_{2n}) \cs,\ct \rangle=&~F\left(\ll (\beta_n\cup \mathbb I_{2n}) \cs,\ct \gg\right)=\\
&=~\ll (\beta_n\cup \mathbb I_{2n}) \cs,\ct \gg |_{y=-d}.
\end{aligned}
\end{equation*}

In the next part we use this graded intersection in order to define two intersection models: one which is associated to the total closure of a braid and another one which corresponds the closure which leaves the first strand open.
\begin{center}
$\cs,\ct \ \ \ \ \ \ \ \ \ \ \ \ \ \ \  \ \ \ \ \ \ \ \ \ \ \ \ \ \ \ \ \ \ \ \ \ \ \ \ \ \ \ \   \ \ \ \ \ \ \ \ \ \ \ \ \cs',\ct'$
\begin{figure}[H]
\centering
\includegraphics[scale=0.3]{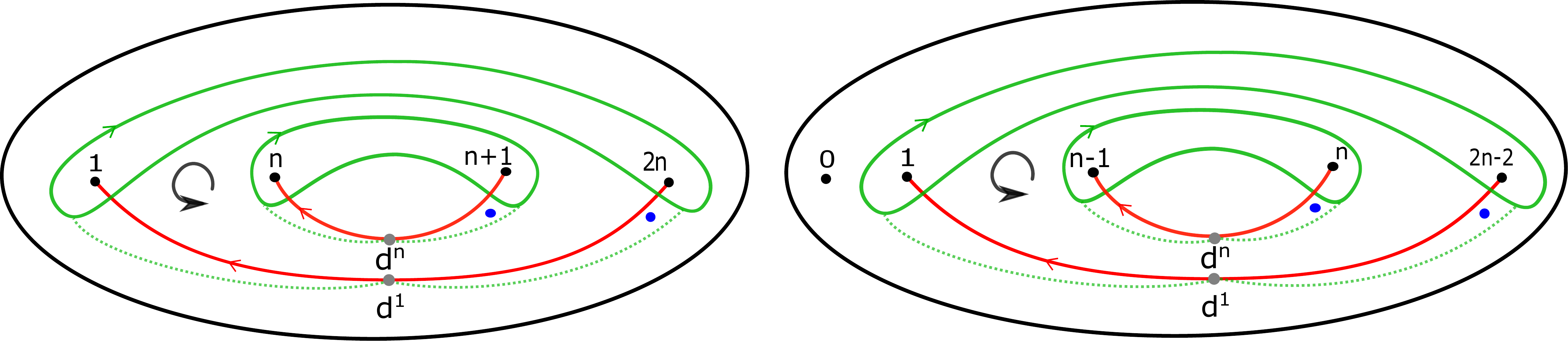}
\hspace{-10mm}\caption{Closed up intersection $\Omega(\beta_n)$ \hspace{30mm} Open intersection $\Omega'(\beta_n)$ \ \ \ \ \ \ \ \ \ \ \ \ \ \ \ \  }\label{Openmodel}
\end{figure}
\end{center}
\begin{defn}(Graded intersections)\label{defn} Let us consider the following polynomials:
$$\Omega(\beta_n)(x,d), \Omega'(\beta_n)(x,d) \in \Z[x^{\pm \frac{1}{2}}, d^{\pm 1}],$$
which are obtained from graded intersections coming from the  the Lagrangian submanifolds form picture \ref{Openmodel}, and are given by the formulas:
\begin{equation}
\begin{aligned}
& \Omega(\beta_n)(x,d):=(d^2x)^{\frac{w(\beta_n)+n}{2}} \cdot d^{-n}\langle  (\beta_n \cup \mathbb I_{2n}) \cs, \ct \rangle\\
& \Omega'(\beta_n)(x,d):=(d^2x)^{\frac{w(\beta_n)+n-1}{2}} \cdot d^{-(n-1)}\langle  (\beta_n \cup \mathbb I_{2n-1}) \cs, \ct \rangle.
\end{aligned}
\end{equation}
We call $\Omega(\beta_n)(x,d)$ the graded intersection associated to the closed up model and $\Omega'(\beta_n)(x,d)$ the graded intersection corresponding to the open model.
\end{defn}
\section{Unknot and the stabilised unknot}\label{S:3}

In this section we investigate the necessary conditions on the ring of coefficients such that the intersection model leads to a link invariant. 
Let us start with the unknot, seen as the closure of the following braids:
$\mathbb I_1\in B_1$ and $ \sigma \in B_2$.
Now we compute the intersection model, which is obtained from the intersection points between the following Lagrangian submanifolds:
\begin{figure}[H]
\centering
\includegraphics[scale=0.4]{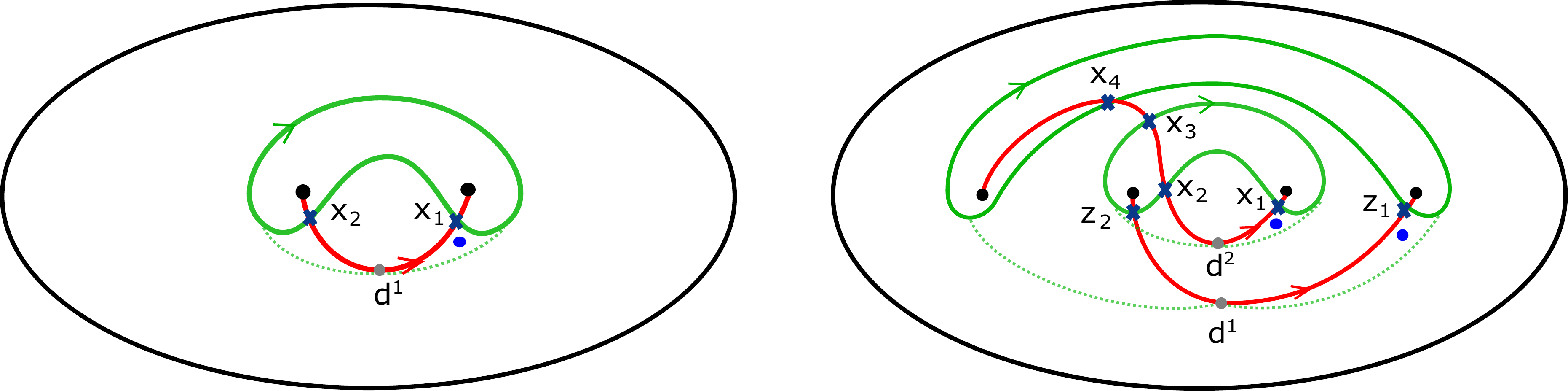}
\caption{Unknot \hspace{50mm} Stabilised unknot}
\end{figure}
Computing the grading of these intersection points from the picture, we obtain the coefficients from below:
\begin{center}
\hspace{15mm}\label{tabel}
{\renewcommand{\arraystretch}{1.7}\begin{tabular}{ | c | c | }
 \hline
$x_1$ & $x_2$\\ 
\hline  
$d$ & $1$\\ 
\hline
\hline                                    
\end{tabular}}
\quad
 \hspace{20mm}{\renewcommand{\arraystretch}{1.7}\begin{tabular}{ | c | c | c | c | }
 \hline
$(x_1,z_1)$ & $(x_2,z_1)$ & $(x_3,z_1)$ & $(x_4,z_2)$  \\ 
\hline  
$d^2$ & $d$ & $-dx^{-1}$ & $d^{-1}x^{-1}$ \\ 
\hline
\hline                                    
\end{tabular}}
\end{center}
This means that the intersection form has the following formulas:
\begin{equation}
\begin{aligned}
&\Omega(\mathbb I_1)=x^{\frac{1}{2}}(1+d)\\
&\Omega(\sigma)=dx^{\frac{3}{2}}(d^{2}+d-dx^{-1}+d^{-1}x^{-1}).
\end{aligned}
\end{equation}
In order to obtain from $\Omega$ a link invariant, this should be the same for these two braids, which give the same knot by braid closure, so the following relation should be true:
\begin{equation}
\Omega(\mathbb I_1)=\Omega(\sigma).
\end{equation}
Using the formulas for the two intersections, we obtain the following relation:
\begin{equation}
d^{2}x+dx-d-1=0
\end{equation}
which is equivalent to:
\begin{equation}
(d+1)(dx-1)=0.
\end{equation}
This shows that $\LL=\Z[x^{\pm \frac{1}{2}}, d^{\pm 1}]/\left( (d+1)(dx-1)\right)$ is the largest quotient in which the intersection form can become a link invariant.
\section{Description of the closed and open graded intersections in terms of state sums}\label{S:2} In order to prove the invariance of our intersection forms with respect to braid conjugation, we will use a different description of $\Omega(\beta_n)$ and $\Omega'(\beta_n)$, as a state sums of Lagrangian intersections following \cite{Cr}. In the next part we outline the construction of these state sum invariants, which use homological representations of braid groups.

\subsubsection{Homological representations}\label{homclasses}

We use the structure of certain Lawrence representations which are presented in \cite{Cr1} and \cite{CrM}. They are braid group representations on subspaces in the homology of a $\mathbb Z \oplus \mathbb Z$-covering of the configuration space $C_{n,m}$, which are generated by explicit homology classes.
For the following definitions we consider the parameter $l=0$.

Let us look at the local system $\phi$ from definition \ref{localsystem}, where we replace the variable $d$ with a variable which we call $d'$ (this is for a sign reason that we will see later on), and denote it by $\phi'$. 
\begin{equation}\label{eq:23}
\begin{aligned}
\phi': \pi_1(C_{n,m}) \rightarrow \ & \Z \oplus \Z\\ 
 & \langle x \rangle \ \langle d' \rangle\\
\end{aligned}
\end{equation}
Let $\tilde{C}_{n,m}$ be the $\Z \oplus \Z$-covering associated to $\phi'$.
\begin{comment}
% The precise definition of these homology groups is given in Section 3 from \cite{Cr2} (Definition 3.1.6 and Proposition 3.1.7). We outline below their main properties, which will be used for our computations.
\begin{itemize}
\item[1)] Lawrence representations $H_{n,m}$ are $\mathbb Z[x^{\pm 1},d^{\pm 1}]$-modules that carry a $B_n$-action.
\item[2)] Dual Lawrence representations $H^{\partial}_{n,m}$ which are $\mathbb Z[x^{\pm 1},d^{\pm 1}]$.
\item[3)] Intersection pairing $ \lll, \ggg : H_{n,m} \otimes H^{\partial}_{n,m}\rightarrow \mathbb Z[x^{\pm 1},d^{\pm1}]$.
\end{itemize}
\end{comment}
Let $w$ be a point on the boundary of the punctured disc and denote by $C_{w}$ the part of the boundary of the configuration space $C_{n,m}$ which is given by configurations containing $w$. Let $\bar{C}_{w}$ be the complement of $C_{w}$ in the boundary of the configuration space.  Then, let $\pi^{-1}({w})$ be the part of the boundary of the covering $\tilde{C}_{n,m}$ given by the fiber over $C_{w}$.

 In the next part we consider part of the Borel-Moore homology of this covering which comes from the Borel-Moore homology of the base space twisted by the local system $\phi'$. 
\begin{prop}(\cite{Cr1})
Let $H^{\text{lf}}_m(\tilde{C}_{n,m},\pi^{-1}(w); \Z)$ be the Borel-Moore homology of the covering relative to part of the boundary given by $\pi^{-1}(w)$, which is a $\Z[x^{\pm1}, d'^{\pm1}]$-module via the group of deck transformations. Then there is a well defined braid group action which is compatible with the structure of a $\Z[x^{\pm1}, d'^{\pm1}]$-module:
$$B_n \curvearrowright H^{\text{lf}}_m(\tilde{C}_{n,m},\pi^{-1}(x); \Z) \ (\text{as a }\Z[x^{\pm1}, d'^{\pm1}]\text{-module}).$$ 
\end{prop}
\begin{notation}
Let $H^{\partial}_{m}(\tilde{C}_{n,m}, \partial^{-}; \Z)$ the homology relative to the boundary of $\tilde{C}_{n,m}$ which is not in the fiber over $w$.

\end{notation}
%\begin{notation} From now on, we will use the variable:$$d:=-d'.$$ \end{notation}
\begin{prop}(\cite{CrM} Theorem E)\label{R:D}
Let $H^{\text{lf}}_m(C_{n,m}, C_{w}; \mathscr L_{\phi})$ and $H_{m}(\tilde{C}_{n,m}, \bar{C}_{w}; \mathscr L_{\phi})$ be the Borel-Moore homology and the homology of the base space relative to the boundary, with coefficients in the local system associated to $\phi'$ (which we denote by $\mathscr L_{\phi}$). Then, we have natural injective maps, which are compatible with the braid group actions as below:
\begin{equation}
\begin{aligned}
&\iota: H^{\text{lf}}_m(C_{n,m}, C_{w}; \mathscr L_{\phi})\rightarrow H^{\text{lf}}_m(\tilde{C}_{n,m}, \pi^{-1}({w});\Z)\\
&\iota^{\partial}: H_{m}(\tilde{C}_{n,m}, \bar{C}_{w}; \mathscr L_{\phi}) \rightarrow H^{\partial}_{m}(\tilde{C}_{n,m}, \partial^{-}; \Z).
\end{aligned}
\end{equation}
\end{prop}
\begin{notation}(Our homology groups)\label{R:1}  
We denote the images of the maps $\iota$ and $\iota^{\partial}$ by:
\begin{equation}
\begin{aligned}
& H_{n,m}\subseteq H^{\text{lf}}_m(\tilde{C}_{n,m}, \pi^{-1}({w});\Z)\\
& H^{\partial}_{n,m}\subseteq H^{\partial}_{m}(\tilde{C}_{n,m}, \partial^{-}; \Z).
\end{aligned}
\end{equation}
\end{notation}
Also, let us consider the following set of partitions:
\begin{equation}
E_{n,m}=\{\bar{j}=(j_1,...,j_{n}) \mid j_1,...,j_{n} \in \Z, \ j_1+...+j_{n}=m \}.
\end{equation}
In the following part we consider a family of homology classes in the above homology groups, which will be given by lifts of submanifolds in the base configuration space. These submanifolds will be encoded by ``geometric supports'' which are collections of curves in the punctured disc. For this part, we fix $d_1,...,d_m$ on the boundary of the punctured disc and ${\bf d}:=(d_1,...,d_m)$ a base point in the configuration space. Moreover, let us fix $\tilde{\bf d}$ to be a lift of this base point in the covering $\tilde{C}_{n,m}$.
\begin{defn}[Homology classes] Let $\bar{j}=(j_1,...,j_{n}) \in E_{n,m}$. The product of ordered configuration spaces on the geometric support from picture \ref{fig3}, whose number of particles is given by the partition $\bar{j}$, quotiented to the unordered configuration space $C_{n,m}$, gives a submanifold: $$\mathbb U_{\bar{j}}\subseteq C_{n,m}.$$
Then, in order to lift this submanifold in the covering, we will use ``paths to the base points'' which are collections of arcs in the punctured disc, from the base points towards the geometric support. More precisely, the collection of dotted paths from the base points towards the arcs from figure \ref{fig3} gives a path in the configuration space $\eta_{\bar{j}}$ from $\bar{d}$ to $\mathbb U_{\bar{j}}$. Then we lift this path to a path $\tilde{\eta}_{\bar{j}}$ through the base point $\tilde{\bf{d}}$. Now, we lift the submanifold $\mathbb U_{\bar{j}}$ to a submanifold $\tilde{\mathbb U}_{\bar{j}}$ through $\tilde{\eta}_{\bar{j}}(1)$. We consider the homology class given by this submanifold and denote it by:
\begin{equation}
\mathscr U'_{j_1,...,j_{n}}:=[\tilde{\mathbb U}_{\bar{j}}] \in  H_{n,m}.
\end{equation}
\end{defn}

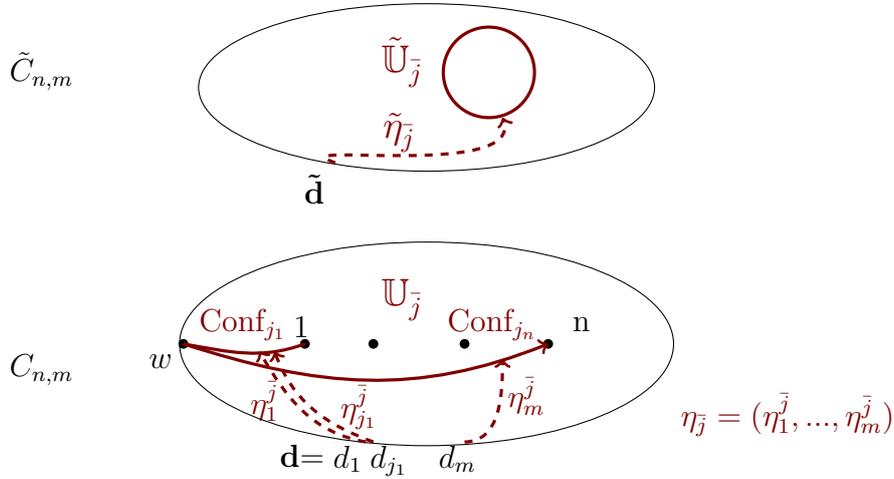
\begin{figure}[H]
\begin{center}
\begin{tikzpicture}\label{pic'}
[x=50 mm,y=10mm,font=\large]
\foreach \x/\y in {-1.2/2, 0.4/2 , 1.3/2 , 2.5/2 , 3.6/2 } {\node at (\x,\y) [circle,fill,inner sep=1.3pt] {};}
\node at (-1.2,2) [anchor=north east] {$w$};
\node at (0.6,2.5) [anchor=north east] {$1$};
\node at (0.2,1.65) [anchor=north east] {$\color{black!50!red}\eta^{\bar{j}}_1$};
\node at (1.5,1.6) [anchor=north east] {$\color{black!50!red} \eta^{\bar{j}}_{j_1}$};
\node at (3.7,1.7) [anchor=north east] {$\color{black!50!red} \eta^{\bar{j}}_{m}$};
\node (dn) at (8.3,1.5) [anchor=north east] {\large \color{black!50!red}$\eta_{\bar{j}}=(\eta^{\bar{j}}_1,...,\eta^{\bar{j}}_m)$};
\node at (2,5.2) [anchor=north east] {\Large \color{black!50!red}$\tilde{\eta}_{\bar{j}}$};

%\node at (1.4,2) [anchor=north east] {i};
%\node at (2.9,2) [anchor=north east] {i+1};
\node at (4.3,2.5) [anchor=north east] {n};
\node at (0.8,0.8) [anchor=north east] {\bf d$=$};
\node at (0.8,4.4) [anchor=north east] {\bf $\bf \tilde{d}$};
\node at (0.3,2.6) [anchor=north east] {$\color{black!50!red}\text{Conf}_{j_1}$};
\node at (3.6,2.6) [anchor=north east] {$\color{black!50!red}\text{Conf}_{j_{n}}$};
\node at (2.1,3) [anchor=north east] {\Large{\color{black!50!red}$\mathbb U_{\bar{j}}$}};
\node at (2.1,6.2) [anchor=north east] {\Large{\color{black!50!red}$\tilde{\mathbb U}_{\bar{j}}$}};
\node at (-2.5,2) [anchor=north east] {\large{$C_{n,m}$}};
\node at (-2.5,6) [anchor=north east] {\large{$\tilde{C}_{n,m}$}};
\draw [very thick,black!50!red,-][in=-160,out=-10](-1.2,2) to (0.4,2);
\draw [very thick,black!50!red,->] [in=-158,out=-18](-1.2,2) to (3.6,2);
 \draw[very thick,black!50!red] (2.82, 5.6) circle (0.6);
\draw (2,2) ellipse (3.25cm and 1.35cm);
\draw (2,5.4) ellipse (3cm and 1.11cm);
\node (d1) at (1.3,0.8) [anchor=north east] {$d_1$};
\node (d2) at (1.9,0.8) [anchor=north east] {$d_{j_1}$};
\node (dn) at (2.8,0.8) [anchor=north east] {$d_m$};
\draw [very thick,dashed, black!50!red,->][in=-60,out=-190](1.2,0.7) to  (-0.2,1.9);
\draw [very thick,dashed,black!50!red,->][in=-70,out=-200](1.3,0.7) to (0,1.9);
\draw [very thick,dashed,black!50!red,->][in=-90,out=0](2.5,0.7) to (3,1.8);
\draw [very thick,dashed,black!50!red,->][in=-70,out=-200](0.8,4.4) to (3,5);
\end{tikzpicture}
\end{center}
\vspace{-6mm}
\caption{Generators for the homology group $H_{n,m}$}
\label{fig3}
\end{figure}
\begin{prop}[Version of the Lawrence representation]\label{gen}
Following \cite{CrM}, this set of homology classes:
\begin{equation}
\mathscr{B}_{H_{n,m}}=\{ \mathscr U'_{j_1,...,j_{n}} \mid j_1,...,j_{n} \in \N, j_1+...+j_{n}=m\} 
\end{equation}
forms a basis for $H_{n,m}$ and there is a braid group action, denoted by:
$$L_{n,m}: B_n\rightarrow \Aut_{\Z[x^{\pm 1},d^{\pm 1}]}\left(H_{n,m}\right).$$
We called this the Lawrence representation.
\end{prop}
\begin{prop}(\cite{CrM}) There is a well defined intersection pairing between the two homology groups as follows:
$$\lll \ , \ \ggg:H_{n,m} \otimes H_{n,m}^{\partial}\rightarrow \Z[x^{\pm 1}, d^{\pm 1}].$$
In this formula $d$ should be thought as $-d'$ and we make this change for computational reasons, as we will see below. 
\end{prop}
This intersection pairing is defined at the level of homology groups, but it can be computed using the geometric supports, the paths to the base points and the local system, in the base configuration space. The precise formula is presented in \cite{Cr1} Proposition 4.4.2. For homology classes which are given by lifts of the geometric supports that we will work with, the formula for the pairing $\lll , \ggg$  at the homological level is actually the same formula as the one for the graded intersection $\ll , \gg$ from definition \ref{int0} in the situation where $l=0$. 
\clearpage
\subsubsection{Our context} Now we present a way to see our Lagrangian intersection through a state sum of intersection pairings between homology clsses belonging to these homology groups. This is based on two Theorems from \cite{Cr}, which we remind below.
Let us fix $n\in \N$. 
\subsubsection{Open model} First, we work with the configuration space $C_{2n-1,n-1}$ and the associated homology groups, where the parameter $k=n-1$. 
%The construction of the homology classes in the covering of the configuration space will be done by drawing a set of disjoint curves in the punctured disc, considering products of ordered configuration spaces on those and then taking their quotient to the unordered configuration space. 
For any multi-index $\bar{i}=(i_1,...,i_{n-1}), i_k \in \{0,1\}, k\in \{1,...,n-1\}$ we look at the two homology classes $\mathscr F'_{\bar{i}} \in H_{2n-1,n-1}, \mathscr L'_{\bar{i}}\in H^{\partial}_{2n-1,n-1}$
obtained by the lifts of the geometric supports together with the paths to the base points from the picture below: 
\begin{center}
${\color{red} \mathscr F'_{\bar{i}} \in H_{2n-1,n-1}} \ \ \ \ \ \ \ \ \ \ \ \ \ \ \  \ \ \ \ \ \ \ \ \ \ \   \ \ \ \ \ \ \ \ \ \ \ \  {\color{dgreen} \mathscr L'_{\bar{i}}\in H^{\partial}_{2n-1,n-1}}.$
\begin{figure}[H]
\centering
\includegraphics[scale=0.4]{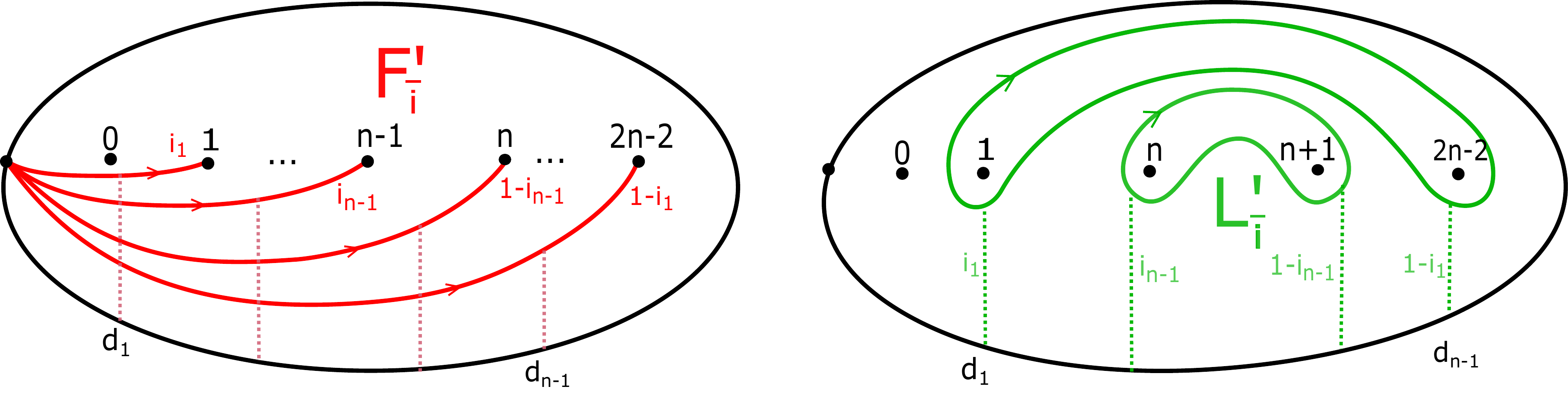}
\caption{State sum model}\label{Statesum'}
\vspace{-25mm}
$$\hspace{15mm} {\color{red}\eta^{F'_{\bar{i}}}} \hspace{53mm} {\color{dgreen}\eta^{L'_{\bar{i}}} }\hspace{14mm}$$
%\vspace{-10mm}
\vspace{2mm}
\end{figure}
\end{center}
\begin{defn}[Specialisations]\label{D:1'''}   Let $c \in \Z$ and consider  the morphism:
\begin{equation}
\begin{aligned}
&\gamma_{c,q,\lambda}: \Z[u^{\pm 1},x^{\pm1},d^{\pm1}]\rightarrow \Z[q^{\pm 1},q^{\pm \lambda}]\\
& \gamma_{c,q,\lambda}(u)= q^{c \lambda}; \ \ \gamma_{c,q,\lambda}(x)= q^{2 \lambda}; \ \ \gamma_{c,q,\lambda}(d)=q^{-2}.
\end{aligned}
\end{equation}
\end{defn}
Using the relations $(5.13),(5.14)$ from the proof of Theorem $5.1$ and Theorem $3.2$ from \cite{Cr}, we have the following model.
\begin{thm}[Unified embedded state sum model \cite{Cr}]\label{Thstate'}

Let $L$ be an oriented link and $\beta_n \in B_n$ such that $L=\hat{\beta}_n$.  Let us consider the polynomial in $3$ variables given by the following state sum:
\begin{equation}\label{Fstate'}
\begin{aligned}
\Lambda'_2(\beta_n)(u,x,d)&:=u^{-w(\beta_n)} u^{-(n-1)} \sum_{i_1,...,i_{n-1}=0}^{1} d^{-\sum_{k=1}^{n-1}i_k} \\
&\lll (\beta_{n} \cup {\mathbb I}_{n-1} ){ \mathscr F'_{\bar{i}}}, {\mathscr L'_{\bar{i}}}\ggg \ \in \Z[u^{\pm1},x^{\pm 1},d^{\pm 1}].
\end{aligned}
\end{equation}
Then we have: 
\begin{equation}\label{eq:4'}
\Omega'(\beta_n)=\Lambda'_2(\beta_n)|_{u=d^{-1}x^{-\frac{1}{2}}}.
\end{equation}
%Moreover, this intersection pairing specialises to the normalised Jones polynomial and the Alexander polynomial of the closure of the braid, as below:
%\begin{equation}\label{eqC:3'}
%\begin{aligned}
%&\Omega'(\beta_n)|_{x=q^{2};d=q^{-2}}=\tilde{J}(L,q)\\
%&\Omega'(\beta_n)|_{d=-1}=\Delta(L,x). 
%\end{aligned}
%\end{equation}
\end{thm}
\begin{coro}(State sum model for the open intersection)\label{CFstate'} As a consequence of the previous formula, we obtain the following sate model for the open intersection form, in the quotient ring:
\begin{equation}
\bar{\Omega}'(\beta_n)(x,d)=(xd^2)^{\frac{w(\beta_n)+(n-1)}{2}} \sum_{i_1,...,i_{n-1}=0}^{1} d^{-\sum_{k=1}^{n-1}i_k} \lll (\beta_{n} \cup {\mathbb I}_{n-1} ){ \mathscr F'_{\bar{i}}}, {\mathscr L'_{\bar{i}}}\ggg\mid_{s}
\end{equation}
Here, the specialisation $s$ is the one defined in \eqref{NN}.
\end{coro}
\clearpage 
\subsubsection{Closed model} For the second model, we work with the configuration space $C_{2n,n}$ and the homology groups associated to the parameter $k=n$. 
For any multi-index $\bar{i}=(i_1,...,i_{n}), i_k \in \{0,1\}, k\in \{1,...,n\}$ we consider the homology classes obtained by the lifts of the following geometric supports: 
\begin{figure}[H]
\centering
${\color{red} \mathscr F_{\bar{i}} \in H_{2n,n}} \ \ \ \ \ \ \ \ \ \ \ \ \ \ \  \ \ \text{ and }\ \ \ \ \   \ \ \ \ \ \ \ \ \ \ \ \   {\color{dgreen} \mathscr L_{\bar{i}}\in H^{\partial}_{2n,n}}.$
\vspace{5mm}

\includegraphics[scale=0.4]{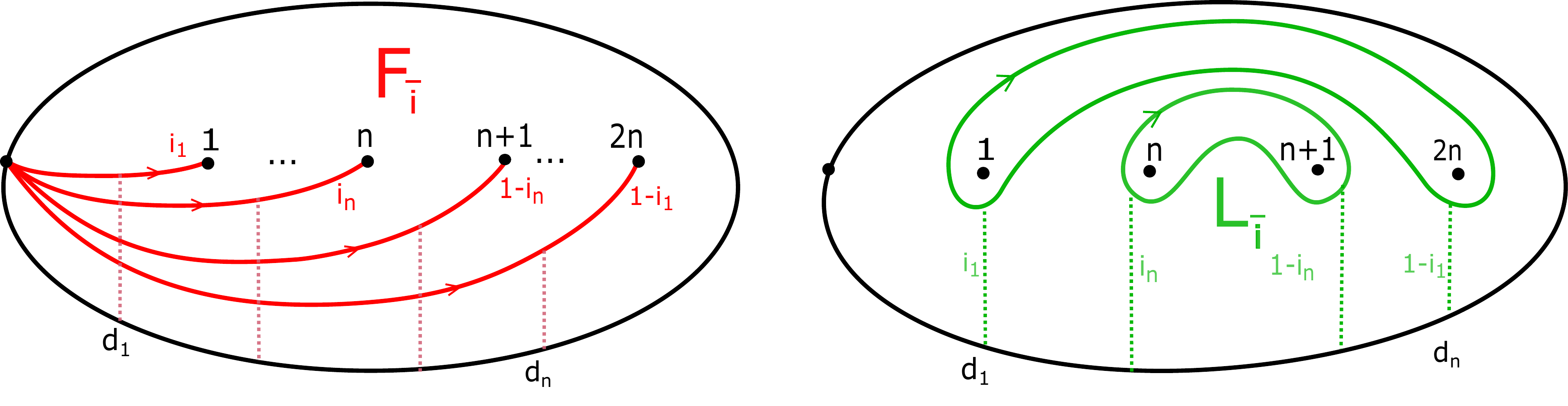}
\caption{State sum model-closed version}\label{Statesum}
\vspace{-25mm}
$$\hspace{15mm} {\color{red}\eta^{F_{\bar{i}}}} \hspace{53mm} {\color{dgreen}\eta^{L_{\bar{i}}} }\hspace{14mm}$$
%\vspace{-10mm}
\vspace{2mm}
\end{figure}

By a similar method as the one used for Theorem \ref{Thstate'}, we deduce a state sum description for the closed intersection model, as follows.
\begin{coro}[Unified embedded state sum model- closed version]\label{Thstate''}

Let $L$ be an oriented link such that $L=\hat{\beta}_n$ for $\beta_n \in B_n$. Let us consider the following state sum:
\begin{equation}\label{Fstate}
\begin{aligned}
\Lambda_2(\beta_n)(u,x,d)&:=u^{-w(\beta_n)} u^{-n} \sum_{i_1,...,i_{n}=0}^{1} d^{-\sum_{k=1}^{n}i_k} \\
&\lll (\beta_{n} \cup {\mathbb I}_{n} ){ \mathscr F_{\bar{i}}}, {\mathscr L_{\bar{i}}}\ggg \ \in \Z[u^{\pm1},x^{\pm 1},d^{\pm 1}].
\end{aligned}
\end{equation}
Then we have:
\begin{equation}\label{eq:4}
\Omega(\beta_n)=\Lambda_2(\beta_n)|_{u=d^{-1}x^{-\frac{1}{2}}}.
\end{equation}
\end{coro}
\begin{coro}[State sum model for the closed intersection]\label{CFstate} From this formula we conclude the following sate model for the closed intersection form, in the quotient ring:
\begin{equation}
\bar{\Omega}(\beta_n)(x,d)=(xd^2)^{\frac{w(\beta_n)+n}{2}} \sum_{i_1,...,i_{n}=0}^{1} d^{-\sum_{k=1}^{n}i_k} \lll (\beta_{n} \cup {\mathbb I}_{n} ){ \mathscr F'_{\bar{i}}}, {\mathscr L'_{\bar{i}}}\ggg\mid_{s} \in \LL.
\end{equation}
\end{coro}
%This intersection corresponds to the quantum trace from the Reshetikhin-Turaev construction of $U_q(sl(2))$-quantum invariants. This means that for generic $q$ it leads to the un-normalized Jones polynomial. However, for $q$ a root of unity, it vanishes, this being a consequence of the vanishing of quantum dimensions at roots of unity. This shows the following property.
\section{Invariance under the Markov moves} \label{S:4}
In this part we prove Theorem \ref{THEOREM} and show that it is enough to quotient towards $\LL$ in order to have a link invariant.
More specifically, we will present a topological proof of the invariance of the closed model $\bar{\Omega}(L)$. We split the proof into two main steps.
\begin{itemize}
\setlength\itemsep{-0.2em}
\item[•] The first step concerns the invariance with respect to the Markov II move. For this, we compute the intersection pairing $\Omega(\beta_n)$ before and after pursuing a stabilisation move, and show that the two formulas become equal if we impose the relation $(d+1)(dx-1)=0$ (which means precisely to consider the quotient morphism and work over $\LL$). 
\item[•]Secondly, we prove that the intersection form is invariant with respect to the Markov I move. We do this by proving that if we pass to the quotient $\LL$, the image $\bar{\Omega}(\beta_n)$ can be interpreted as a sum of traces of braid group representations. Then we conclude that this sum is invariant with respect to braid conjugation.
\end{itemize}

\subsection{Markov II} In this part we want to prove that the intersection $\bar{\Omega}$ is invariant under stabilisations:
\begin{equation}
\bar{\Omega}(\beta_n)=\bar{\Omega}(\sigma^{\pm1}_n\circ \beta_n).
\end{equation}
We will do this by checking this move via the state sum $\Lambda_2$. More precisely, we compute which relation is needed in order to obtain the same result before and after the stabilisation. Following relation \eqref{Fstate} we have: 
\begin{equation}\label{equations}
\begin{aligned}
&\Lambda_2(\beta_n)(u,x,d)=u^{-w(\beta_n)} u^{-n} \sum_{i_1,...,i_{n}=0}^{1} d^{-\sum_{k=1}^{n}i_k} \lll (\beta_{n} \cup {\mathbb I}_{n} ){ \mathscr F'_{\bar{i}}}, {\mathscr L'_{\bar{i}}}\ggg.\\
&\Lambda_2(\sigma^{\pm1}_n\circ\beta_n)(u,x,d)=u^{-w(\sigma^{\pm1}_n\circ\beta_n)} u^{-(n+1)} \sum_{i_1,...,i_{n}=0}^{1} d^{-\sum_{k=1}^{n}i_k}\cdot \\
&\hspace{15mm}\cdot \left(\lll (\sigma^{\pm1}_n\circ\beta_{n} \cup {\mathbb I}_{n+1} ){ \mathscr F'_{\bar{i},0}}, {\mathscr L'_{\bar{i},0}}\ggg
+ d^{-1}\lll (\sigma^{\pm1}_n\circ\beta_{n} \cup {\mathbb I}_{n+1} ){ \mathscr F'_{\bar{i},1}}, {\mathscr L'_{\bar{i},1}}\ggg \right).
\end{aligned}
\end{equation}
In the next part, for a set of indices $\bar{j}=(j_1,...,j_n)$ we denote their sum by: 
$$w(\bar{j}):=j_1+...+j_n.$$
Now we fix an index $\bar{i}$ and we look at the terms from the above state sums that are associated to this index. Let us denote by $m=w(\bar{i})$. From the structure of the homology group $H_{2n,m}$, presented in Proposition \ref{gen}, there exists a collection of coefficients $\alpha_{\bar{j}}\in \Z[x^{\pm1}, d^{\pm1}]$ such that:
\begin{equation}\label{coeff2}
(\beta_{n} \cup {\mathbb I}_{n} ){ \mathscr F'_{\bar{i}}}=\sum_{\substack{\bar{j}=(j_1,...,j_{n})\in E_{n,m}}}\alpha_{\bar{j}} \cdot \mathscr U'_{\bar{j},1-\bar{i}}
\end{equation}
(here we used Notation \ref{ind}).

Then, in the first state sum, the term associated to the index $\bar{i}$ can be expressed as:
\begin{equation}
d^{-m}\sum_{\substack{\bar{j}\in E_{n,m}}} \alpha_{\bar{j}} \cdot \lll \mathscr U'_{\bar{j},1-\bar{i}}, {\mathscr L'_{\bar{i}}}\ggg.
\end{equation}
For the next part we are interested in the intersection with the dual class.
\begin{prop} We have the following property of the intersection form:
\begin{equation}\label{prop}
\hspace{-3mm}\lll  \mathscr U'_{\bar{j},1-\bar{i}}, {\mathscr L'_{\bar{i}}}\ggg=\begin{cases}
1, \text{ if } (j_1,...,j_{n})=(i_1,...,i_{n})\\
0, \text{ otherwise}.
\end{cases}
\end{equation}
\end{prop}
The proposition follows by an analog argument as the one from Lemma 7.7.1 from \cite{Cr1}. This shows that if
$$\lll \mathscr U'_{\bar{j},1-\bar{i}}, {\mathscr L'_{\bar{i}}}\ggg\neq 0$$
then we have to have $\bar{j}=\bar{i}$. So, in the first state sum, associated to the index $\bar{i}$ we have:
\begin{equation}\label{1}
\alpha_{\bar{i}} \cdot \lll \mathscr U'_{\bar{i},1-\bar{i}}, {\mathscr L'_{\bar{i}}}\ggg.
\end{equation}

For the next part we look at the classes from the second state sum (which are associated to a set of $(n+1)$ indices). Using relation \eqref{coeff2} we have the decomposition from below:
\begin{equation}\label{coeff3}
\left((\beta_{n}\cup{\mathbb I})\cup {\mathbb I}_{n+1}\right){ \mathscr F'_{\bar{i},\epsilon}}=\sum_{\substack{\bar{j}=(j_1,...,j_{n})\in E_{n,m}}}\alpha_{\bar{j}} \cdot \mathscr U'_{\bar{j},\epsilon,1-\epsilon,1-\bar{i}}, \forall \epsilon \in \{0,1\}.
\end{equation}
Then, in the second state sum, the term associated to the index $\bar{i}$ has the following formula:
\begin{equation}
d^{-m}\sum_{\substack{\bar{j}\in E_{n,m}}} \alpha_{\bar{j}} \cdot \left(\lll (\sigma^{\pm1}_n\cup {\mathbb I}_{n+1}) \mathscr U'_{\bar{j},0,1,1-\bar{i}}, {\mathscr L'_{\bar{i},0}}\ggg + d^{-1} \lll (\sigma^{\pm1}_n\cup {\mathbb I}_{n+1}) \mathscr  U'_{\bar{j},1,0,1-\bar{i}}, {\mathscr L'_{\bar{i},1}}\ggg \right).
\end{equation}
We remark that the $(\sigma^{\pm1}_n\cup {\mathbb I}_{n+1})$-action on $\mathscr U'_{\bar{j},0,1,1-\bar{i}}$ will be a linear combination of classes which correspond to indices that have the first components $j_1,...,j_{n-1}$. Just the indices which are associated to the $n^{th}$ and $(n+1)^{st}$ positions can be modified. 
Since we are intersecting with the dual class $\mathscr L'_{\bar{i},0}$, using the property from relation \eqref{prop}, we conclude that the above intersections give a non-trivial term just in the situation where:
\begin{equation}
(j_1,...,j_{n-1})=(i_1,...,i_{n-1}).
\end{equation}
Since $j_1+...+j_n=i_1+...+i_n=m$, we conclude that actually the two indexing sets have to coincide:
\begin{equation}
(j_1,...,j_{n})=(i_1,...,i_{n}), \text{ and so } \bar{j}=\bar{i}.
\end{equation}
As a conclusion, in the second state sum, corresponding to the index $\bar{i}$ we have:
\begin{equation}\label{2}
\alpha_{\bar{i}} \cdot \left(\lll (\sigma^{\pm1}_n\cup {\mathbb I}_{n+1}) \mathscr  U'_{\bar{i},0,1,1-\bar{i}}, {\mathscr L'_{\bar{i},0}}\ggg + d^{-1} \lll (\sigma^{\pm1}_n\cup {\mathbb I}_{n+1}) \mathscr  U'_{\bar{i},1,0,1-\bar{i}}, {\mathscr L'_{\bar{i},1}}\ggg \right).
\end{equation}
Putting together relations \eqref{equations}, \eqref{1} and \eqref{2} we conclude that the invariance at the Markov II move is true if for any $\bar{i}=(i_1,...,i_n)$ with $i_1,...,i_n\in \{0,1\}$:
\begin{equation}
\begin{aligned}
&\lll \mathscr U'_{\bar{i},1-\bar{i}}, {\mathscr L'_{\bar{i}}}\ggg=\\
&=u^{\mp-1} \cdot \left(\lll (\sigma^{\pm1}_n\cup {\mathbb I}_{n+1}) \mathscr  U'_{\bar{i},0,1,1-\bar{i}}, {\mathscr L'_{\bar{i},0}}\ggg + d^{-1} \lll (\sigma^{\pm1}_n\cup {\mathbb I}_{n+1}) \mathscr  U'_{\bar{i},1,0,1-\bar{i}}, {\mathscr L'_{\bar{i},1}}\ggg \right).
\end{aligned}
\end{equation}
On the other hand, all the coefficients that appear at the intersections $$\lll (\sigma^{\pm1}_n\cup {\mathbb I}_{n+1}) \mathscr  U'_{\bar{i},0,1,1-\bar{i}}, {\mathscr L'_{\bar{i},0}}\ggg \text { and }\lll (\sigma^{\pm1}_n\cup {\mathbb I}_{n+1}) \mathscr  U'_{\bar{j},1,0,1-\bar{i}}, {\mathscr L'_{\bar{i},1}}\ggg$$ come from the intersection points which belong to the two inner green circles, all the other points contribute by coefficients which are $1$. 

From this remark, we conclude that is enough to check the Markov II move for braids with two strands (which correspond to $n=1$). This means that the following condition should be satisfied:
\begin{equation}
\begin{aligned}
&\lll \mathscr U'_{i,1-i}, {\mathscr L'_{i}}\ggg=\\
&=u^{\mp-1} \cdot \left(\lll (\sigma^{\pm1}\cup {\mathbb I}_{2}) \mathscr  U'_{i,0,1,1-i}, {\mathscr L'_{i,0}}\ggg + d^{-1} \lll (\sigma^{\pm1}\cup {\mathbb I}_{2}) \mathscr  U'_{i,1,0,1-i}, {\mathscr L'_{i,1}}\ggg \right).
\end{aligned}
\end{equation}
for any $i\in \{0,1\}$.
So, we have two conditions:
\begin{equation}\label{conditions}
\hspace{-3mm}\begin{cases}
\begin{aligned}
\lll \mathscr U'_{0,1},& {\mathscr L'_{0}}\ggg=\\
&=u^{\mp-1}\left(\lll (\sigma^{\pm}\cup {\mathbb I}_{2}) \mathscr  U'_{0,0,1,1}, {\mathscr L'_{0,0}}\ggg + d^{-1} \lll (\sigma^{\pm1}\cup {\mathbb I}_{2}) \mathscr  U'_{0,1,0,1}, {\mathscr L'_{0,1}}\ggg \right).\\
\lll \mathscr U'_{1,0},& {\mathscr L'_{1}}\ggg=\\
&=u^{\mp-1}\left(\lll (\sigma^{\pm1}\cup {\mathbb I}_{2}) \mathscr  U'_{1,0,1,0}, {\mathscr L'_{1,0}}\ggg + d^{-1} \lll (\sigma^{\pm1}\cup {\mathbb I}_{2}) \mathscr  U'_{1,1,0,0}, {\mathscr L'_{1,1}}\ggg \right).
\end{aligned}
\end{cases}
\end{equation}
For the left hand side of the above equations, we have the intersections from below:
\begin{figure}[H]
\centering
$$ \ \ \ \ \ \lll \mathscr U'_{0,1}, {\mathscr L'_{0}}\ggg=1 \ \ \ \ \ \ \ \ \ \ \ \ \ \ \ \ \ \ \ \ \ \  \lll \mathscr U'_{1,0}, {\mathscr L'_{1}}\ggg=1 \ \ \ \ \ \ $$
\hspace{-3mm}\includegraphics[scale=0.4]{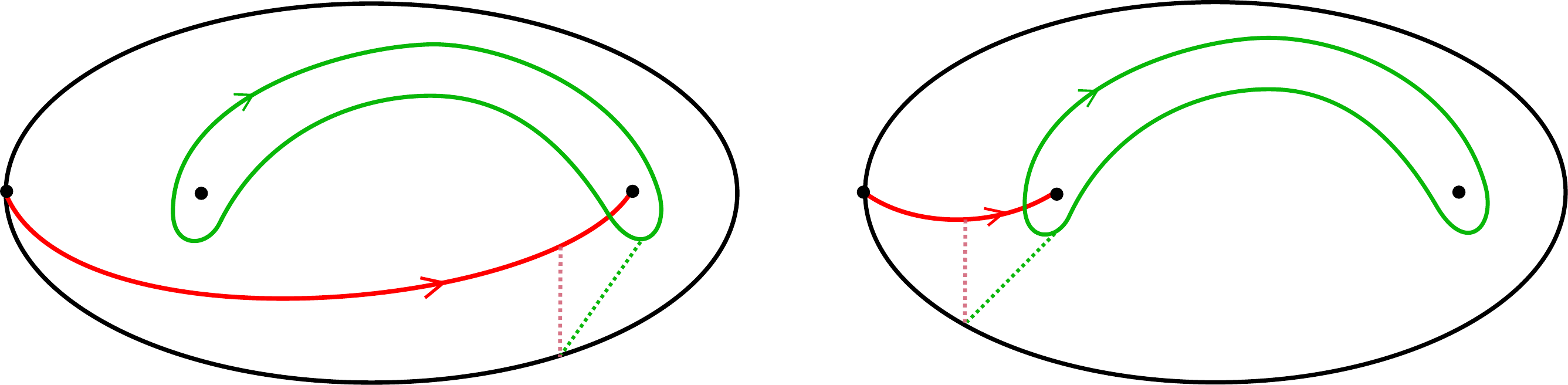}
\caption{Coefficients before the stabilisation}
\end{figure}
In the following part we investigate the conditions from relation \eqref{conditions} for two cases, given by positive stabilisation or the negative stabilisation.  
\subsubsection{Positive stabilisation} We compute the intersections which appear in condition \eqref{conditions} for the case when we act with the positive generator $\sigma$. We have the following intersections:
\begin{figure}[H]
\centering
$$\lll (\sigma\cup {\mathbb I}_{2}) \mathscr  U'_{0,0,1,1}, {\mathscr L'_{0,0}}\ggg=1 \ \ \ \ \ \ \ \ \ \ \ \ \ \ \ \ \  \lll (\sigma\cup {\mathbb I}_{2}) \mathscr  U'_{1,0,1,0}, {\mathscr L'_{1,0}}\ggg=0$$
\includegraphics[scale=0.3]{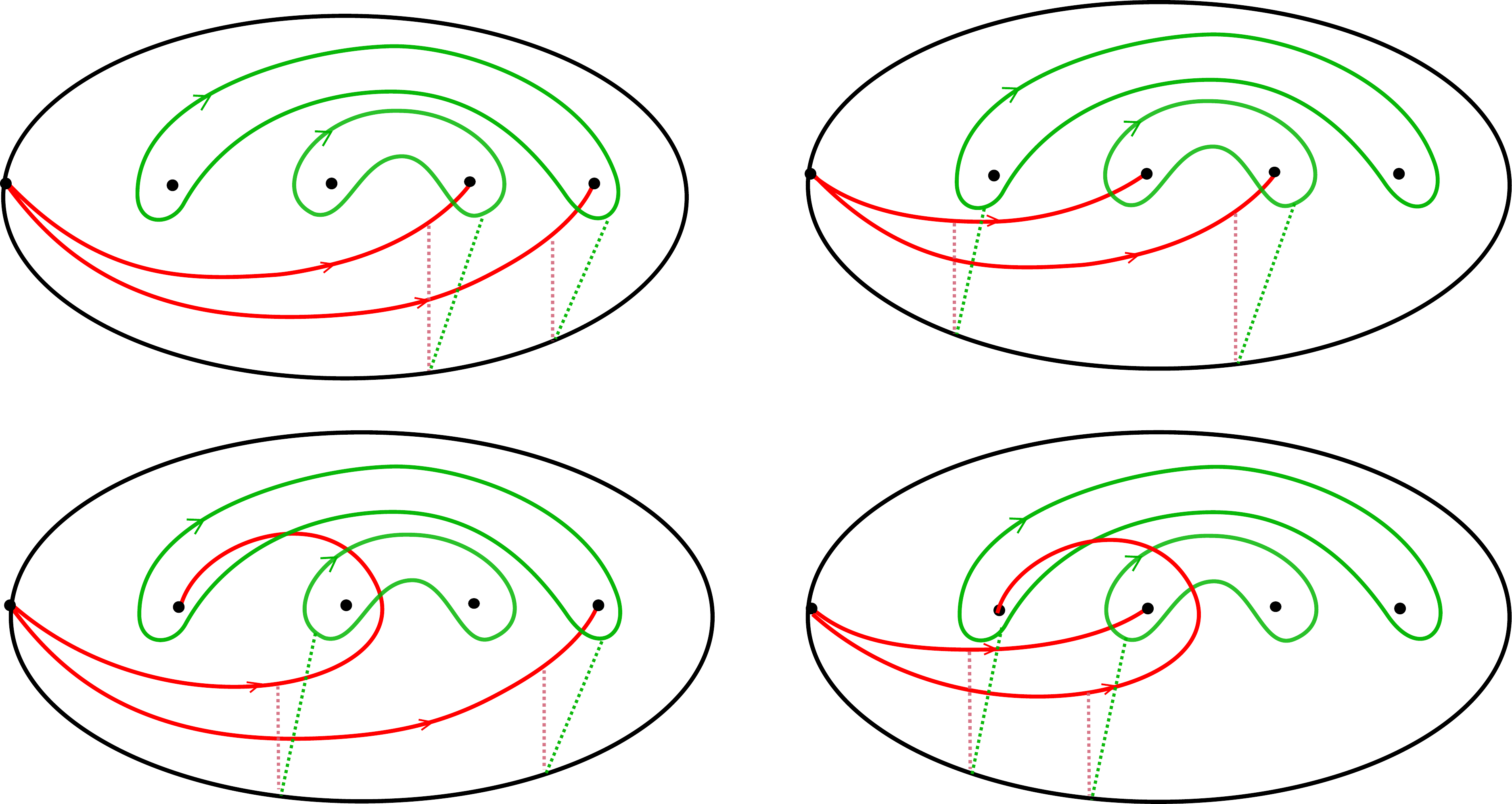}
$$\lll (\sigma\cup {\mathbb I}_{2}) \mathscr  U'_{0,1,1,0}, {\mathscr L'_{0,1}}\ggg=1-x^{-1} \ \ \ \ \ \ \ \ \ \ \ \ \ \ \ \ \  \lll (\sigma \cup {\mathbb I}_{2}) \mathscr  U'_{1,1,0,0}, {\mathscr L'_{1,1}}\ggg=d^{-1}x^{-1}$$
\caption{Coefficients of the positive stabilisation}
\end{figure}
We obtain the system:
\begin{equation}
\begin{cases}
1=u^{-2}\left(1+d^{-1}(1-x^{-1})\right)\\
1=u^{-2}d^{-2}x^{-1}.
\end{cases}
\end{equation}
This is equivalent to:
\begin{equation}
\begin{cases}
1=d^2x\left(1+d^{-1}(1-x^{-1})\right)\\
u^{-2}=d^2x.
\end{cases}
\end{equation}
Then, the first condition becomes:
\begin{equation}
\begin{aligned}
1=d^2x+dx-d & \Leftrightarrow \ \ d^2x-1+dx-d=0 \ \ \Leftrightarrow \ \ dx(d+1)-(d+1)=0 \ \ \Leftrightarrow\\
& \Leftrightarrow (d+1)(dx-1)=0.
\end{aligned}
\end{equation}
This is precisely the condition that we have in the quotient ring, so this equation is satisfied.
For the second one, we choose $u=d^{-1}x^{-\frac{1}{2}}$, which is precisely the specialisation used in relation \eqref{eq:4}. 

\subsubsection{Negative stabilisation} In this part we investigate the invariance at the negative stabilisation and check relations \eqref{conditions} in this situation. We have the coefficients given by the following intersections:
\begin{figure}[H]
\centering
$$\lll (\sigma^{-1}\cup {\mathbb I}_{2}) \mathscr  U'_{0,0,1,1}, {\mathscr L'_{0,0}}\ggg=1 \ \ \ \ \ \ \ \ \ \ \ \ \ \ \ \ \  \lll (\sigma^{-1}\cup {\mathbb I}_{2}) \mathscr  U'_{1,0,1,0}, {\mathscr L'_{1,0}}\ggg=1-x$$
\includegraphics[scale=0.3]{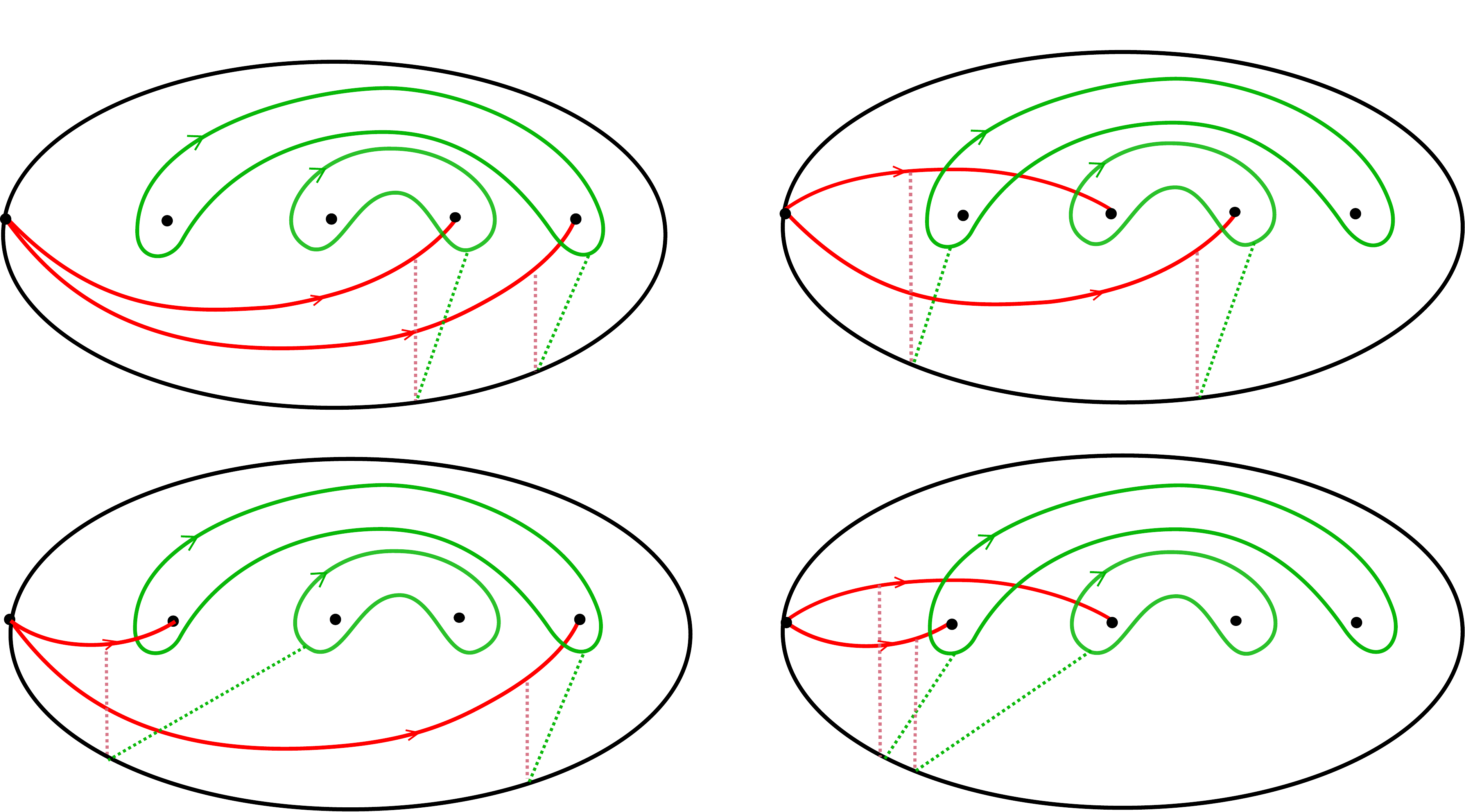}
$$\lll (\sigma^{-1}\cup {\mathbb I}_{2}) \mathscr  U'_{0,1,1,0}, {\mathscr L'_{0,1}}\ggg=0 \ \ \ \ \ \ \ \ \ \ \ \ \ \ \ \ \  \lll (\sigma^{-1} \cup {\mathbb I}_{2}) \mathscr  U'_{1,1,0,0}, {\mathscr L'_{1,1}}\ggg=dx$$
\caption{Coefficients of the negative stabilisation}
\end{figure}
We obtain the relations:
\begin{equation}
\begin{cases}
1=1+d^{-1}\cdot 0\\
1=1-x+d^{-1}(dx).
\end{cases}
\end{equation}
These are both true, so we conclude that the negative stabilisation is always satisfied (even before we quotient to the quotient ring). 

This concludes the invariance of the intersection $\bar{\Omega}(\beta_n)$ with respect to the second Markov move. For the intersection form $\Omega'(\beta_n)$ we can pursue an analog argument, this time having $n-1$ particles in the configuration space, and conclude in a similar way that $\bar{\Omega}'(\beta_n)$ is invariant at stabilisations.

\subsection{Markov I}In this part we aim to prove the invariance of the form $\bar{\Omega}$ with respect to braid conjugation. Let $\beta_n,\gamma_n \in B_n$. We want to show that:
\begin{equation}\label{eq:3}
\bar{\Omega}\left(\beta_n\right)(x,d)=\bar{\Omega}\left(\gamma_n \circ \beta_n \circ \gamma^{-1}_n\right)(x,d).
\end{equation}
We follow the formula presented in Corollary \ref{Thstate''}, use the intersection form $\Lambda_2(\beta_n)(u,x,d)$ and prove that after we take the quotient it becomes invariant under conjugation. We notice that the writhe and number of strands remain unchanged under conjugation, so the framing part from $\Lambda_2(\beta_n)(u,x,d)$ (which is given by the power of $u$) is invariant under conjugation. We will show that if we impose the condition $(d-1)(xd-1)=0$ then $\Lambda_2(\beta_n)(u,x,d)$ is invariant under conjugation.

Our strategy is to prove that this state sum specialised by the above condition can be seen as a sum of traces of braid group representations, which in turn are conjugacy invariants. 

We start by introducing the following subspace in the Lawrence representation.
\begin{defn}(Subspace in the Lawrence representation) Let us consider the set of partitions from $E_{n,m}$ with multiplicities at most one:
\begin{equation}
E^{1}_{n,m}=\{\bar{j}=(j_1,...,j_{n}) \mid j_1,...,j_{n} \in \Z, j_1+...+j_{n}=m, 0 \leq j_1,...,j_{n} \leq 1\}.
\end{equation}
Then, we consider the subspace $H^{1}_{n,m}\subseteq H_{n,m}$ generated by classes which are prescribed by such partitions, as below:
\begin{equation}
H^{1}_{n,m}=\langle \mathscr U'_{j_1,...,j_{n}} \mid \bar{j}=(j_1,...,j_{n}) \in E^{1}_{n,m}\rangle_{\Z[x^{\pm1},d^{\pm1}]}. 
\end{equation}
\end{defn}
As we will see in the next part, this subspace is not preserved by the braid group action given by the Lawrence representation on $H_{n,m}$. However, if we impose the extra relation, then the subspace will be preserved under the $B_n$-action and we will have a well defined sub-representation.
\begin{lem}[Sub-representation of the Lawrence representation] We consider the quotient morphism
$s:\Z[x^{\pm1},d^{\pm1}]\rightarrow\Z[x^{\pm1},d^{\pm1}]/((d-1)(xd-1))$ from \eqref{NN}.
Then there is a well defined induced representation of the braid group on the following subspace:
\begin{equation}
B_n\curvearrowright H^{1}_{n,m}|_{s}
\end{equation}
(here, we use Notation \ref{N:spec}).
\end{lem}
\begin{proof}
We have to prove that for all $\bar{j}=(j_1,...,j_n)\in E^{1}_{n,m}$ and any $\beta_n\in B_n$ we have:
$$\beta_n \mathscr U'_{j_1,...,j_{n}} \in H^{1}_{n,m}|_{s}.$$
It is enough to show this for the generators of the braid group and moreover, since the action of such generator is local and acts non-trivially just on a disc around two punctures, it is enough to check this in that punctured disc with two punctures. 
Let $j_0,j_1\in \{0,1\}$ and denote $j_0+j_1=m$. Then we will prove that: $$\sigma \mathscr U'_{j_0,j_{1}} \in H^{1}_{2,m}|_{s}.$$ 
If $j_1=0$, looking directly on the picture we see that we obtain another basis element associated to a partition without multiplicities. The only check that needs to be done is for $j_1=1$.
The only case when we could get a multiplicity at least $2$ is if $j_0=1$. Using the structure of $H_{2,m}$, we know that in this homology group we have a decomposition:
\begin{equation}\label{eq:40}
\sigma \mathscr U'_{1,1}=\alpha_1 \mathscr  U'_{1,1}+ \alpha_2 \mathscr U'_{0,2} + \alpha_3 \mathscr U'_{2,0}. 
\end{equation}
\begin{figure}[H]
\centering
\includegraphics[scale=0.3]{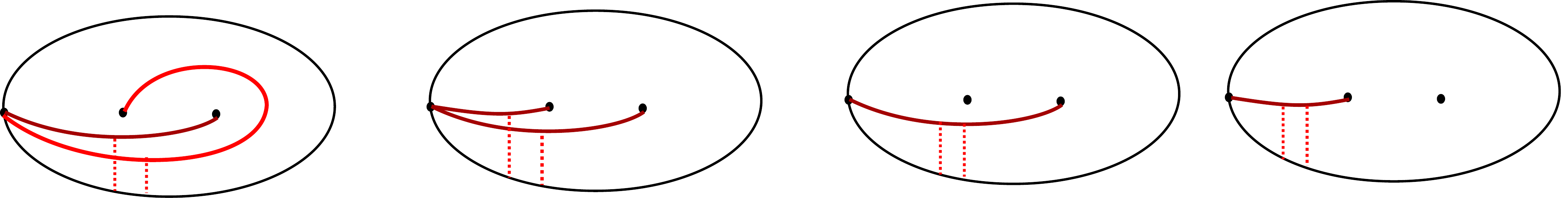}
\vspace{-3mm}
$$\sigma \mathscr U'_{1,1} \ \ \hspace{24mm} \ \mathscr  U'_{1,1} \ \ \hspace{24mm} \ \mathscr U'_{0,2} \ \ \hspace{24mm} \ \mathscr U'_{2,0}
$$
\caption{Coefficients}
\end{figure}
Now, intersecting with a dual class whose support has two semi-circles which start from the upper boundary of the disc and go around the first puncture, we see that all intersections vanish except the one with the class $\mathscr U'_{2,0}$-which is $1$. This shows that the coefficient $\alpha_3=0$. 

In the next part we compute the coefficient $\alpha_2$. We do this by intersecting with the dual class $\mathscr V$ given by the following green barcode:
\begin{figure}[H]
\centering
\includegraphics[scale=0.3]{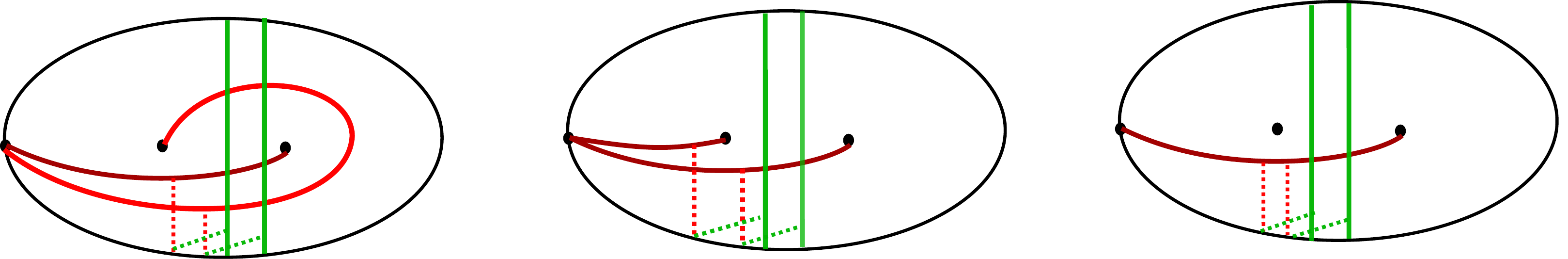}
\vspace{-3mm}
$$ \ \ \hspace{54mm} \ 0 \ \ \hspace{29mm} \ 1 \ \ \hspace{24mm} 
$$
\caption{Computing $\alpha_2$}\label{A}
\end{figure}
Following relation \eqref{eq:40}, we have:
\begin{equation}
\lll \sigma \mathscr U'_{1,1},\mathscr V \ggg =  \alpha_1 \lll \mathscr  U'_{1,1},\mathscr V \ggg+ \alpha_2 \lll \mathscr U'_{0,2},\mathscr V \ggg. 
\end{equation}
From the pictures from Figure \ref{A}, we see that:
\begin{equation}
\begin{aligned}
\lll \mathscr  U'_{1,1},\mathscr V \ggg=0 \\
\lll \mathscr  U'_{0,2},\mathscr V \ggg=1.
\end{aligned}
\end{equation}
Now we compute the intersection $\lll \sigma \mathscr U'_{1,1},\mathscr V \ggg$. 
\begin{figure}[H]
\centering
\includegraphics[scale=0.7]{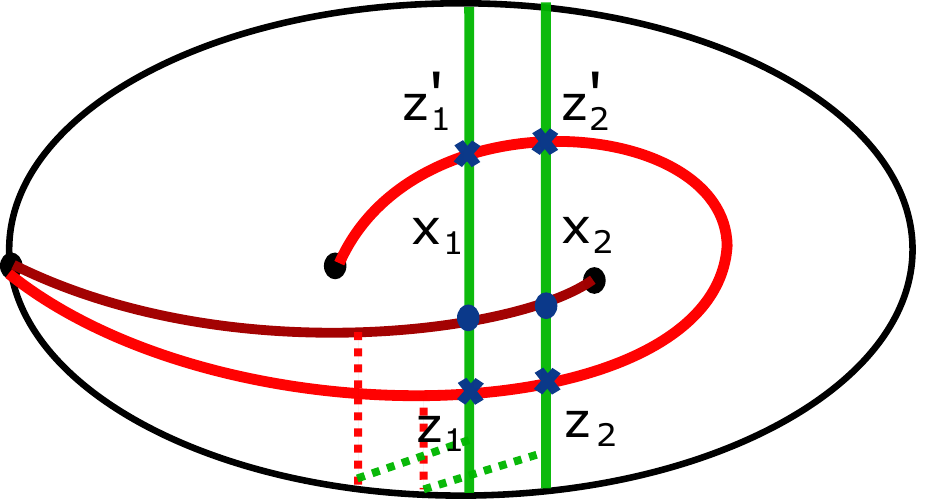}
\caption{Coefficients $\sigma$}\label{AA}
\end{figure}
For this, we use Figure \ref{AA}, where we see that we have 4 intersection points, which carry the following gradings:
\begin{center}
\label{tabel}
{\renewcommand{\arraystretch}{1.7}\begin{tabular}{ | c | c | c | c | }
 \hline
$(x_1,z_2)$ & $(x_1,z'_2)$ & $(x_2,z_1)$ & $(x_2,z'_1)$  \\ 
\hline  
$1$ & $-x^{-1}$ & $d$ & $-x^{-1}d^{-1}$ \\ 
\hline
\hline                                    
\end{tabular}}
\end{center}
It follows that the intersection is:
\begin{equation}
\lll \sigma \mathscr U'_{1,1},\mathscr V \ggg=1+d-(x^{-1}+x^{-1}d^{-1})=(1+d)(1-x^{-1}d^{-1}).
\end{equation}
From this, we obtain the coefficient $\alpha_2$:
\begin{equation}
\alpha_2=(1+d)(1-x^{-1}d^{-1}).
\end{equation}

In order to compute the coefficient $\alpha_1$, we intersect with the barcode  $\mathscr W$ from the picture below:
\begin{figure}[H]
\centering
\includegraphics[scale=0.7]{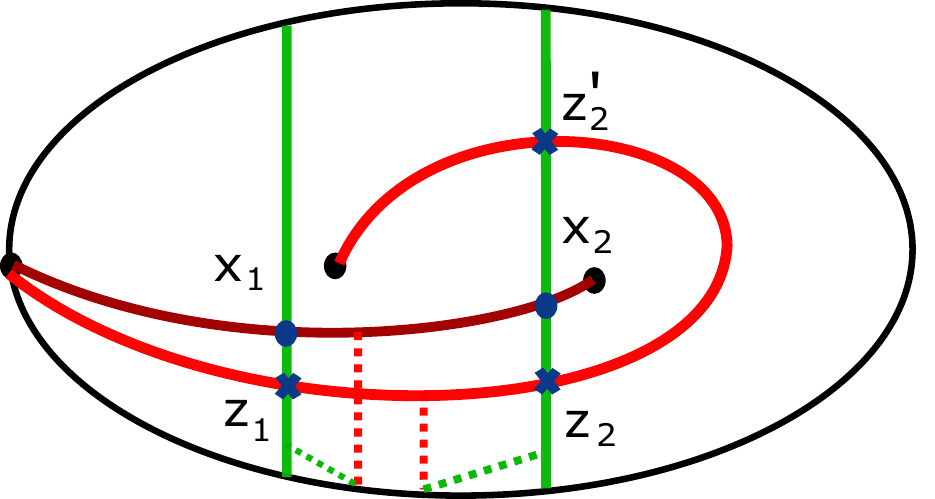}
\caption{Computing $\alpha_1$}
\end{figure}
We remark that it has the following intersections with the basis elements:
\begin{equation}
\begin{aligned}
\lll \mathscr  U'_{1,1},\mathscr W \ggg=1 \\
\lll \mathscr  U'_{0,2},\mathscr W \ggg=0.
\end{aligned}
\end{equation}
For the intersection $\lll \sigma \mathscr U'_{1,1},\mathscr W \ggg$ we have 3 intersection points which carry the gradings from below:
\begin{center}
{\renewcommand{\arraystretch}{1.7}\begin{tabular}{ | c | c | c | c | }
 \hline
$(x_1,z_2)$ & $(x_1,z'_2)$ & $(x_2,z_1)$   \\ 
\hline  
$1$ & $-x^{-1}$ & $d$  \\ 
\hline
\hline                                    
\end{tabular}}
\end{center}

It follows that $\lll \sigma \mathscr U'_{1,1},\mathscr W \ggg=1-x^{-1}+d$, and using the decomposition \eqref{eq:40} we conclude that:
\begin{equation}
\alpha_1=1-x^{-1}+d.
\end{equation}
So, we have the decomposition:
\begin{equation}
\sigma \mathscr U'_{1,1}=(1-x^{-1}+d) \mathscr  U'_{1,1}+ (1+d)(1-x^{-1}d^{-1}) \mathscr U'_{0,2}. 
\end{equation}
It follows that if we impose the relation $(1+d)(1-x^{-1}d^{-1})$, the coefficient $\alpha_2$ vanishes. This shows that $\sigma \mathscr U'_{1,1} \in H^1_{2,2}|_s$, so we remain in the homological subspace given by multiplicity free partitions. 

In the next part we do the same procedure for the action of the elementary braid $\sigma^{-1}$. Similar to the previous case, the only situation that we need to check is given by the action on the class $\mathscr U'_{1,1}$ and we have a decomposition as below:
\begin{equation}\label{eq:41}
\sigma^{-1} \mathscr U'_{1,1}=\alpha'_1 \mathscr  U'_{1,1}+ \alpha'_2 \mathscr U'_{0,2} + \alpha'_3 \mathscr U'_{2,0}. 
\end{equation}
It is clear that we have the coefficient $\alpha'_2=0$, because we cannot get a geometric support with multiplicity two that ends in the second puncture. So, we have the following classes that appear in the decomposition:
\begin{figure}[H]
\centering
\includegraphics[scale=0.3]{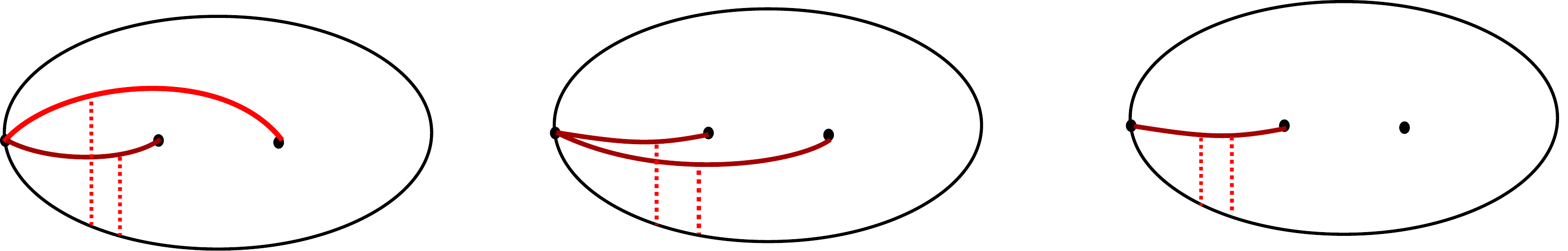}
\vspace{-3mm}
$$\sigma^{-1} \mathscr U'_{1,1} \ \ \hspace{24mm} \ \mathscr  U'_{1,1} \ \ \hspace{24mm} \ \mathscr U'_{2,0}
 \ \ \ \  $$
\caption{Coefficients $\sigma^{-1}$}
\end{figure}
First of all, we want to compute $\alpha_1'$. In order to do this, we intersect with the barcode $\mathscr W$:
\begin{figure}[H]
\centering
\includegraphics[scale=0.3]{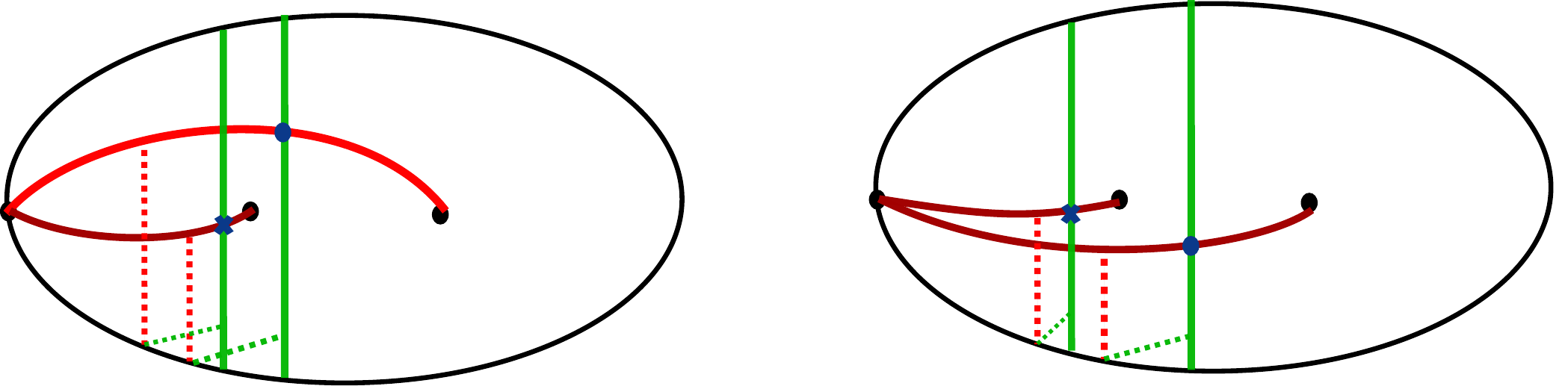}
\vspace{-3mm}
$$ \ \ \hspace{18mm} \ xd \ \ \hspace{29mm} \ 1 \ \ \hspace{24mm} 
$$
\caption{Computing $\alpha'_1$}
\end{figure}
Computing the coefficients from the two intersections, we obtain:
\begin{equation}
\begin{aligned}
& \lll \sigma^{-1} \mathscr U'_{1,1},\mathscr W \ggg=xd \\
& \lll \mathscr  U'_{1,1},\mathscr W \ggg=1 \\
& \lll \mathscr  U'_{2,0},\mathscr W \ggg=0.
\end{aligned}
\end{equation}
From relation \eqref{eq:41} we get:
\begin{equation}
\alpha'_1=xd.
\end{equation}
Further on, we want to find $\alpha_3'$. We intersect with the following barcode $\mathscr Z$:
\begin{figure}[H]
\centering
\includegraphics[scale=0.3]{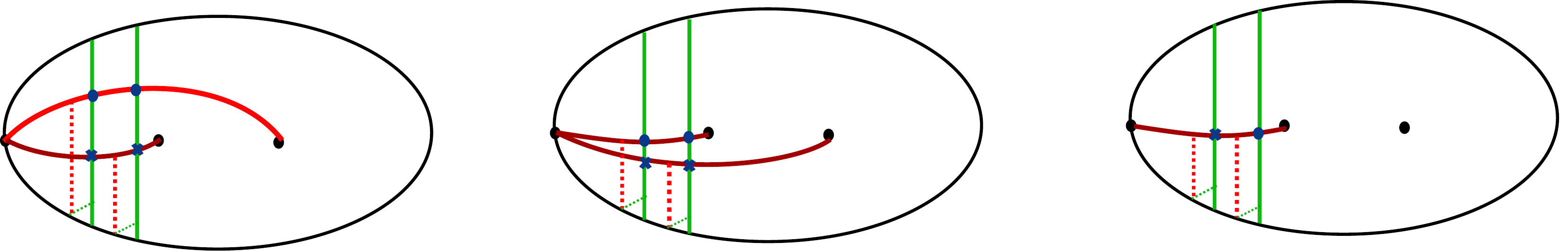}
\vspace{-3mm}
$$\hspace{17mm} 1+d \ \ \hspace{24mm} \ 1+d \ \ \hspace{27mm} \ 1 \ \ \hspace{20mm} 
$$
\caption{Computing $\alpha'_3$}
\end{figure}
We have the following intersections:
\begin{equation}
\begin{aligned}
& \lll \sigma^{-1} \mathscr U'_{1,1},\mathscr Z \ggg=xd \\
& \lll \mathscr  U'_{1,1},\mathscr Z \ggg=1+d \\
& \lll \mathscr  U'_{2,0},\mathscr Z \ggg=1.
\end{aligned}
\end{equation}

This together with the decomposition \eqref{eq:41} show the following relation:
\begin{equation*}
1+d=(1+d)\alpha'_1+\alpha'_3.
\end{equation*}
Using that $\alpha'_1=xd$ we conclude:
\begin{equation}
\alpha'_3=1+d-(1+d)xd=(1+d)(1-xd).
\end{equation}
We obtain the following decomposition:
\begin{equation}
\sigma^{-1} \mathscr U'_{1,1}= xd \ \mathscr  U'_{1,1}+ (1+d)(1-xd) \ \mathscr U'_{2,0}. 
\end{equation}

 This shows that if we pass to the quotient and impose the quadratic relation, then we remain in the subspace: $$\sigma^{-1} \mathscr U'_{1,1} \in H^1_{2,2}|_s.$$ This concludes the proof that the subspace $H_{n,m}^1$ remains invariant under the braid group action once we specialise the coefficients via the function $s$.
 \end{proof}
 \begin{notation}[Lawrence sub-representation]
We denote this well defined sub-representation by:
\begin{equation}
L^1_{n,m}: B_n\rightarrow \Aut\left(H^{1}_{n,m}|_{s}\right).
\end{equation}
\end{notation}

Now we are ready to show that the $s$-specialised intersection form is invariant under conjugation. We remind the formula from Corollary \ref{CFstate}:
\begin{equation*}
\begin{aligned}
\bar{\Omega}(\beta_n)(x,d)=(xd^2)^{\frac{w(\beta_n)+n}{2}} \sum_{i_1,...,i_{n}=0}^{1} d^{-\sum_{k=1}^{n}i_k} \lll (\beta_{n} \cup {\mathbb I}_{n} ){ \mathscr F'_{\bar{i}}}, {\mathscr L'_{\bar{i}}}\ggg\mid_{s}\ \in \LL.
\end{aligned}
\end{equation*}
In the next part, we will show that the above state sum can be interpreted using the Lawrence sub-representations.
\begin{prop} The state sum of intersections is a sum of traces of Lawrence sub-representations, as below:
\begin{equation}\label{traceformula}
\begin{aligned}
\sum_{i_1,...,i_{n}=0}^{1} d^{-\sum_{k=1}^{n}i_k}
\lll (\beta_{n} \cup {\mathbb I}_{n} ){ \mathscr F'_{\bar{i}}}, {\mathscr L'_{\bar{i}}}\ggg|_{s} ~=\sum_{m=0}^{n} d^{-m} \  tr(L^1_{n,m}(\beta_n)).
\end{aligned}
\end{equation}
\end{prop}

\begin{proof}
The state sum from the left hand side can be expressed as:
\begin{equation}\label{relstate}
\begin{aligned}
&\sum_{i_1,...,i_{n}=0}^{1} d^{-\sum_{k=1}^{n}i_k}
 \lll (\beta_{n} \cup {\mathbb I}_{n} ){ \mathscr F'_{\bar{i}}}, {\mathscr L'_{\bar{i}}}\ggg=\\
&=\sum_{m=0}^{n} d^{-m} \sum_{\bar{i}=(i_1,...,i_{n})\in E_{n,m}^1}
 \lll (\beta_{n} \cup {\mathbb I}_{n} ){ \mathscr F'_{\bar{i}}}, {\mathscr L'_{\bar{i}}}\ggg.
\end{aligned}
\end{equation}
Now, using the structure of the homology group $H_{2n,m}$ from proposition \ref{gen}, for $\bar{i}\in E_{n,m}^1$ there exists a collection of coefficients $\alpha_{\bar{j}} \in \Z[x^{\pm1},d^{\pm1}]$ such that:
\begin{equation}\label{coeff1}
 L_{2n,m}(\beta_{n} \cup {\mathbb I}_{n} ){ \mathscr F'_{\bar{i}}}=\sum_{\substack{\bar{j}=(j_1,...,j_{n})\in E_{n,m}}}\alpha_{\bar{j}} \cdot \mathscr U'_{\bar{j},1-\bar{i}}
\end{equation}
(following notation \ref{ind}). 

This comes from the fact that on the last components we act with ${\mathbb I}_{n}$, so we do not change the associated indices of $\mathscr F'_{\bar{i}}$ through this action and so they remain $1-\bar{i}$. On the other hand, we will obtain a linear combination of classes associated to partitions whose first components are $\bar{j}$ for arbitrary $\bar{j}=(j_1,...,j_n)$. Since the Lawrence representation preserve the total sum of indices, it follows that we will get classes associated to $\bar{j}$ such that $$w(\bar{j})=w(\bar{i})=m.$$This explains relation \eqref{coeff1}.
For the intersection with the dual class, we remind relation \eqref{prop}: 
\begin{equation}
\hspace{-3mm}\lll  \mathscr U'_{\bar{j},1-\bar{i}}, {\mathscr L'_{\bar{i}}}\ggg=\begin{cases}
1, \text{ if } (j_1,...,j_{n})=(i_1,...,i_{n})\\
0, \text{ otherwise}.
\end{cases}
\end{equation}
 So, the pairing with the dual class $\mathscr L'_{\bar{i}}$ encodes precisely the diagonal coefficient and for any $\bar{i}\in E_{n,m}^{1}$ we have:
\begin{equation}\label{diagonalelement}
\begin{aligned}
\lll (\beta_{n} \cup {\mathbb I}_{n} ){ \mathscr F'_{\bar{i}}}, {\mathscr L'_{\bar{i}}}\ggg ~=\sum_{\substack{\bar{j}= (j_1,...,j_{n})\in E_{n,m}}}\alpha_{\bar{j}} \lll \mathscr U'_{\bar{j},1-\bar{i}}, \mathscr L'_{\bar{i}}\ggg~=\alpha_{\bar{i}} .
\end{aligned}
\end{equation}
On the other hand, we remark that these $\alpha$-coefficients are the same as the ones that give the decomposition of the $\beta_n$-action on the basis element $\mathscr U'_{\bar{i}}$ from the homology group $H_{n,m}$:
\begin{equation}
L_{n,m}(\beta_{n}){ \mathscr U'_{\bar{i}}}=\sum_{\substack{\bar{j}=(j_1,...,j_{n})\in E_{n,m}}}\alpha_{\bar{j}} \cdot \mathscr U'_{\bar{j}}.
\end{equation}
We notice that the above sum is indexed by elements from $E_{n,m}$, not necessarily from $E^{1}_{n,m}$.

From the relation \eqref{diagonalelement} we see that for any index $\bar{i}\in E_{n,m}^{1}$ the pairing $$\lll (\beta_{n} \cup {\mathbb I}_{n} ){ \mathscr F'_{\bar{i}}}, {\mathscr L'_{\bar{i}}}\ggg$$ encodes precisely the coefficient of the homology class $\mathscr U'_{\bar{i}}$ that appear in the decomposition of $L_{n,m}(\beta_{n}){ \mathscr U'_{\bar{i}}}$. 

Now we specialise through $s$ and use the property that we have a well defined action $L^1_{n,m}$ on the subspace $H^1_{n,m}$ (which is spanned by all $\mathscr U'_{\bar{j}}$ for $\bar{j}\in E_{n,m}^{1}$). From these remarks we obtain that:
\begin{equation}
\begin{aligned}
\sum_{\bar{i}=(i_1,...,i_{n})\in E_{n,m}^1}
 \lll (\beta_{n} \cup {\mathbb I}_{n} ){ \mathscr F'_{\bar{i}}}, {\mathscr L'_{\bar{i}}}\ggg|_{s}&=\sum_{\bar{i}=(i_1,...,i_{n})\in E_{n,m}^1} \alpha_{\bar{i}} \ =\\
 &= tr \left(L^1_{n,m}(\beta_n)\right).
\end{aligned}
\end{equation}
This together with relation \eqref{relstate} conclude the trace formula interpretation for the state sum, as presented in the statement.
\end{proof}
From this, we obtain a trace formula for our intersection form.
\begin{coro}The specialised intersection form $\bar{\Omega}(\beta_n)$ is given by the following sum of traces of Lawrence sub-representations:
\begin{equation}
\begin{aligned}
\bar{\Omega}(\beta_n)(x,d)&:=(xd^2)^{\frac{w(\beta_n)+n}{2}} \sum_{m=0}^{n} d^{-m} \  tr(L^1_{n,m}(\beta_n)).
\end{aligned}
\end{equation}

\end{coro}
Using the property that the trace is invariant under conjugation, we conclude that when we impose relation $(1+d)(xd-1)$ the intersection form $\bar{\Omega}$ is invariant at conjugation, so the first Markov move is satisfied. 

The invariance of $\bar{\Omega}$ with respect to the two Markov moves shows that it is a well-defined link invariant with values in $\LL$, and concludes Theorem \ref{THEOREM}.
\subsection{Markov I move for specialisations of the open model $\Omega'(\beta_n)$} 
This subsection concerns the invariance of two specialisations of the form $\bar{\Omega}'(x,d)$ with respect to braid conjugation.  All the proofs so far were topological. This is the only place where we use one result from to our previous work. More precisely, in \cite{Cr} we have proved that the open intersection model recovers the Jones and Alexander polynomials of the closure of the braid, through the following specialisations of coefficients:
\begin{equation}\label{eq:11}
\begin{aligned}
&\Omega'(\beta_n)(x,d)|_{x=d^{-1}}=\tilde{J}(\hat{\beta}_n,x)\\
&\Omega'(\beta_n)|_{d=-1}=\Delta(\hat{\beta}_n,x).
\end{aligned}
\end{equation}
We will re-prove this in the next section by checking the skein relation. However, for that we will use the property that these two specialisations are invariant under braid conjugation, namely: 
\begin{equation}\label{eq:conj}
\begin{aligned}
&\Omega'(\beta_n)(x,d)|_{x=d^{-1}}=\Omega'(\gamma \circ \beta_n \circ \gamma^{-1})(x,d)|_{x=d^{-1}}\\
&\Omega'(\beta_n)(x,d)|_{x={-1}}=\Omega'(\gamma \circ \beta_n \circ \gamma^{-1})(x,d)|_{x={-1}}.
\end{aligned}
\end{equation}
for any braid $\gamma \in B_n$.
\section{Identification of the specialisations of the intersection form with Jones and Alexander invariants, via skein relations}\label{S:7}

In this part we aim to prove that the specialisations of the intersection forms $\bar{\Omega}'$ and $\bar{\Omega}$ recover the Jones and Alexander polynomials, as presented in Theorem \ref{Tsk} and Theorem \ref{Tsk'}. We do this by checking that they satisfy the skein relations that characterise these two polynomials. 

Let $L$ be the closure of $\beta_n$ and suppose that we want to investigate a crossing change. We denote by $L_+$ and $L_-$ the two links obtained from $L$ by performing a positive or negative crossing change respectively. We have proved that the intersection form $\bar{\Omega}$ is a link invariant, in particular it is invariant under conjugation. Also, from \eqref{eq:conj} we know that $\Omega'$ is also conjugation invariant. So, for both intersection models we can suppose that the crossing change is performed at the top of the braid. Let $i$ and $i+1$ be the two adjacent strands that form this crossing. 

Let us denote the two braids obtained from $\beta_n$ by adding a positive/ negative crossing as:
\begin{equation} 
\begin{aligned}
&\beta_n^{+}=({\mathbb I}_{i-1} \cup \sigma_i \cup {\mathbb I}_{n-i-1}) \circ \beta_n;\\ &\beta_n^{-}=({\mathbb I}_{i-1} \cup \sigma_i^{-1} \cup {\mathbb I}_{n-i-1}) \circ \beta_n.
\end{aligned}
\end{equation}

We will use an argument which is similar to the one presented in section \ref{S:4}, which permits us to check the skein computation just for two strands.

\subsection{Open intersection model}
We consider the graded intersections which are associated to the crossing change and following Corollary \ref{CFstate'} we have:
\begin{equation}\label{sk}
\begin{aligned}
\bar{\Omega}'(\beta_n)(x,d)&=(xd^2)^{\frac{w(\beta_n)+(n-1)}{2}} \sum_{i_1,...,i_{n-1}=0}^{1} d^{-\sum_{k=1}^{n-1}i_k} \lll (\beta_{n} \cup {\mathbb I}_{n-1} ){ \mathscr F'_{\bar{i}}}, {\mathscr L'_{\bar{i}}}\ggg\mid_{s}\\
\bar{\Omega}'(\beta^{+}_n)(x,d)&=(x^{\frac{1}{2}}d)(xd^2)^{\frac{w(\beta_n)+(n-1)}{2}} \sum_{i_1,...,i_{n-1}=0}^{1} d^{-\sum_{k=1}^{n-1}i_k}\\
& \hspace{20mm} \lll (({\mathbb I}_{i-1} \cup \sigma_i \cup {\mathbb I}_{n-i-1}) \cup {\mathbb I}_{n-1} )(\beta_{n} \cup {\mathbb I}_{n-1} ){ \mathscr F'_{\bar{i}}}, {\mathscr L'_{\bar{i}}}\ggg\mid_{s}\\
\bar{\Omega}'(\beta^{-}_n)(x,d)&=(x^{-\frac{1}{2}}d^{-1})(xd^2)^{\frac{w(\beta_n)+(n-1)}{2}} \sum_{i_1,...,i_{n-1}=0}^{1} d^{-\sum_{k=1}^{n-1}i_k}\\
& \hspace{20mm} \lll (({\mathbb I}_{i-1} \cup \sigma_i \cup {\mathbb I}_{n-i-1}) \cup {\mathbb I}_{n-1} )(\beta_{n} \cup {\mathbb I}_{n-1} ){ \mathscr F'_{\bar{i}}}, {\mathscr L'_{\bar{i}}}\ggg\mid_{s}.
\end{aligned}
\end{equation}
For a multi-index $\bar{i}$, we have a common part which appears in all these intersections, namely: $$(\beta_{n} \cup {\mathbb I}_{n-1} ){ \mathscr F'_{\bar{i}}}|_{s}.$$
Let us denote the sum of the components of this multi-index by $w(\bar{i})=m$. Using a similar decomposition as the one presented in relation \eqref{coeff1} for $n-1$ instead of $n$ and the specialisation $s$, we remark that this homology class decomposes as a sum of homology classes with total weight $(n-1)$:
\begin{equation}
 (\beta_{n} \cup {\mathbb I}_{n-1} ){ \mathscr F'_{\bar{i}}|_{s}}=L^1_{2n-1,m}(\beta_{n} \cup {\mathbb I}_{n-1} ){ \mathscr F'_{\bar{i}}}=\sum_{\substack{\bar{j}=(j_1,...,j_{n})\in E^1_{n,m}}}\alpha_{\bar{j}} \cdot \mathscr U'_{\bar{j},1-\bar{i}}.
\end{equation}
Then, in order to compute $\bar{\Omega}'(\beta^{+}_n)$ and $\bar{\Omega}'(\beta^{-}_n)$, for each such $\bar{j}\in E^1_{n,m}$ we have to act with $\sigma^{\pm1}_i$, which gives the following homology class:
\begin{equation}
({\mathbb I}_{i-1} \cup \sigma_i^{\pm1} \cup {\mathbb I}_{n-i-1})\mathscr U'_{\bar{j},1-\bar{i}}.
\end{equation}
 This class will be a linear combination of homology classes which have the same indices as $\mathscr U'_{\bar{j},1-\bar{i}}$ except possibly the indices which are located in positions $i$ and $i+1$. Thus, in the three intersection forms that characterise a crossing change, the only difference occurs when we compute the intersections associated to the strands $i$ and $i+1$ with the $i$ and $i+1$ green circles (all the other terms are evaluated by the same scalar when we intersect with the dual class).

Based on this, we conclude that it would be enough to check that the skein relation holds for the intersections associated to any any multi-index $\bar{i}=(i_1,i_2)$ where $i_1,i_2 \in \{0,1\}$. They arise from the following geometric supports: 
\begin{itemize}
\item[•] I) $\mathscr U'_{0,0}$
\item[•] II) $\mathscr U'_{1,0}$
\item[•] III) $\mathscr U'_{0,1}$
\item[•] IV) $\mathscr U'_{1,1}$.
\end{itemize}
We are going to check the skein relation for each of these four cases. We will be using the models with arcs and circles, that give the intersections $\Omega'(\beta_n)$, $\Omega'(\beta^{+}_n)$ and $\Omega'(\beta^{-}_n)$. In this geometric picture, the four cases can be distinguished based on the number of components of an intersection point that are chosen on the right hand side of the picture. More precisely, each intersection point will have two components in the punctured disc. In picture \ref{Skein}, we denote the points in the punctured disc that belong to the right hand side of the disc by $x_0$ and $y_0$. The four cases correspond to the following situations:
\begin{itemize}
\item[•] I) intersection points which have both components $x_0$ and $y_0$  (so it is the uniqe point $(x_0,y_0)$).
\item[•] II) intersection points which have as component $x_0$ but not $y_0$.  
\item[•] III) intersection points which have as component $y_0$ but not $x_0$.
\item[•] IV) intersection points which do not have as component $y_0$ nor $x_0$.
\end{itemize}
 In figure \ref{Skein}, we present the three pictures that are associated to the intersections $\Omega(\mathbb I_2)$, $\Omega(\sigma)$ and $\Omega(\sigma^{-1})$. In the tables which are next to each picture we have: the intersection points, their associated gradings and we also indicate on the first row which of the four cases they belong to.
\begin{figure}[H]
\centering
$$\hspace{85mm} \bar{\Omega}'(\beta^{+}_n) $$
\hspace{78mm}{\renewcommand{\arraystretch}{1.7}\begin{tabular}{ | c | c | c | c | }
 \hline
I & II & II & IV  \\ 
\hline                                    
\hline
$(x_0,y_0)$ & $(x_0,y_1)$ & $(x_0,y_2)$ & $(x_1,y_3)$  \\ 
\hline  
$d^2$ & $d$ & $-dx^{-1}$ & $d^{-1}x^{-1}$ \\ 
\hline
\end{tabular}}
\vspace{10mm}

\includegraphics[scale=0.38]{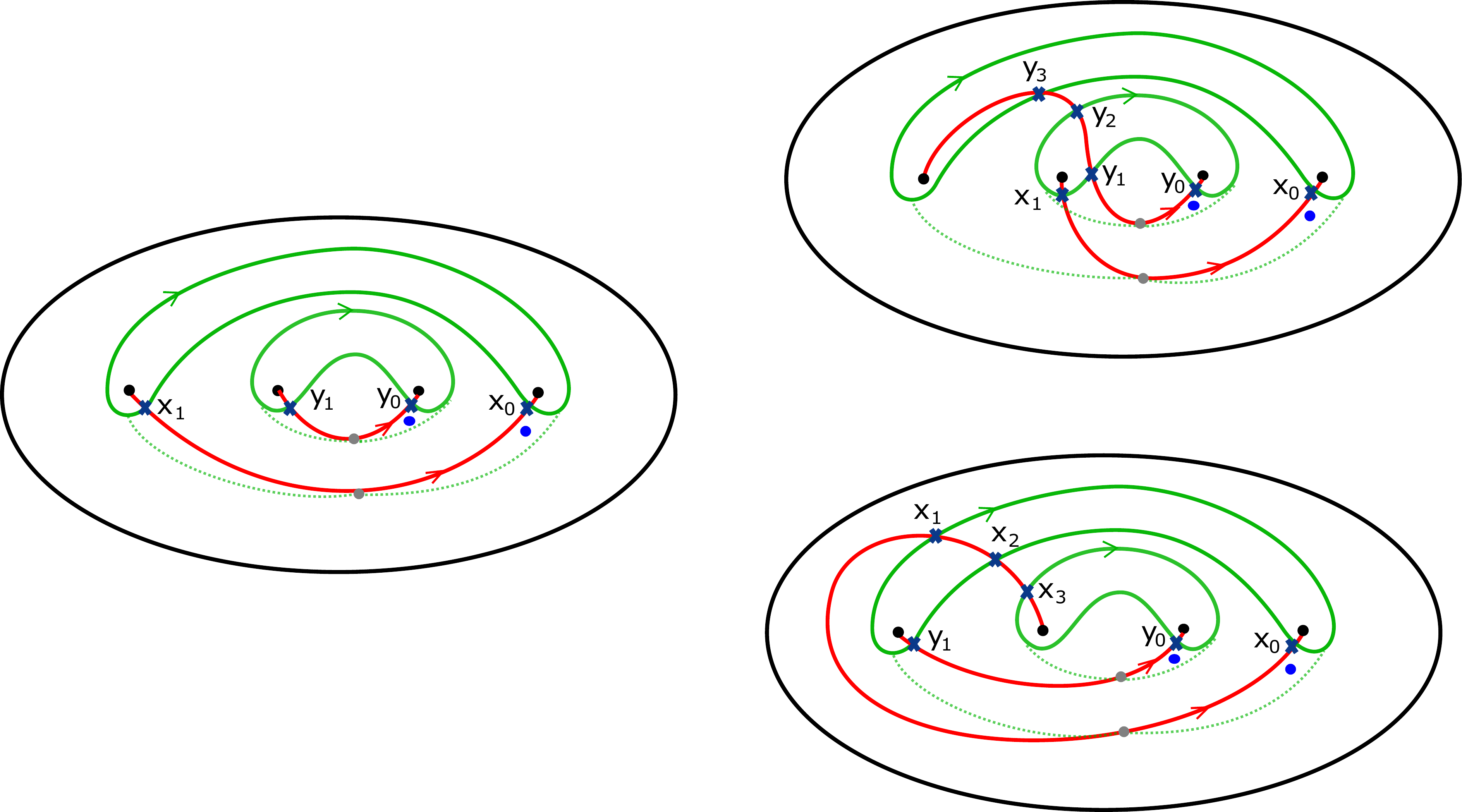}

\vspace{-27mm}
$$\hspace{-80mm} \bar{\Omega}'(\beta_n) $$
\vspace{15mm}
$$\hspace{85mm} \bar{\Omega}'(\beta^{-}_n) $$
\begin{center}
{\renewcommand{\arraystretch}{1.7}\begin{tabular}{ | c | c | c | c | }
 \hline
I & II & III & IV  \\ 
\hline                                    
\hline
$(x_0,y_0)$ & $(x_0,y_1)$ & $(x_1,y_0)$ & $(x_1,y_1)$  \\ 
\hline  
$d^2$ & $d$ & $d$ & $1$ \\ 
\hline
\end{tabular}}
\hspace{15mm}{\renewcommand{\arraystretch}{1.7}\begin{tabular}{ | c | c | c | c | }
 \hline
I & III & III & IV  \\ 
\hline                                    
\hline
$(x_0,y_0)$ & $(x_1,y_0)$ & $(x_2,y_0)$ & $(x_3,y_1)$  \\ 
\hline  
$d^2$ & $d$ & $-dx$ & $dx$ \\ 
\hline
\end{tabular}}
\end{center}
\hspace{-10mm}\caption{Skein relation}\label{Skein}
\end{figure}
Taking into account the extra framing contributions, which are $x^{\frac{1}{2}}d$ for $\bar{\Omega}'(\beta^{+}_n)$ and $x^{-\frac{1}{2}}d^{-1}$ for $\bar{\Omega}'(\beta^{-}_n)$, we obtain the following coefficients:
\begin{table}[H]
\hspace{10mm}{\renewcommand{\arraystretch}{1.9}\begin{tabular}{| C{2.5cm} | C{2cm} | C{2cm} | C{2cm} | C{2cm} | }
 \hline
& I & II & II & IV  \\ 
\hline                                    
\hline
$\bar{\Omega}'(\beta_n)$ &  $d^2$ & $d$ & $d$ & $1$  \\ 
\hline  
$\bar{\Omega}'(\beta^{+}_n)$&  $d^3x^{\frac{1}{2}}$ & $d^2(x^{\frac{1}{2}}-x^{-\frac{1}{2}})$ & & $x^{-\frac{1}{2}}$  \\ 
\hline  
$\bar{\Omega}'(\beta^{-}_n)$&  $dx^{-\frac{1}{2}}$ &  & $-(x^{\frac{1}{2}}-x^{-\frac{1}{2}})$ & $x^{\frac{1}{2}}$  \\ 
\hline  
\end{tabular}}
\caption{\label{Skein1}Coefficients of the intersections $\bar{\Omega}'(\beta_n), \bar{\Omega}'(\beta^{+}_n)$, $\bar{\Omega}'(\beta^{-}_n)$}
\end{table}
Let us denote:
$$\nu:=x^{\frac{1}{2}}-x^{-\frac{1}{2}}.$$
\subsubsection{Jones polynomial} Let us look at the specialisation $d=x^{-1}$ of these three intersection forms (denoted by $\psi_J$, as in Definition \ref{N}). We check the skein relation where we multiply the coefficients associated to $\beta_n^{+}$ by $x$, the ones associated to $\beta_n^{+}$ by $x^{-1}$ and the ones for $\beta_n$ by $\nu$. Doing this using Table \ref{Skein1}, we obtain the following coefficients:
\begin{table}[H]
\hspace{10mm}{\renewcommand{\arraystretch}{1.9}\begin{tabular}{| C{2.5cm} | C{2cm} | C{2cm} | C{2cm} | C{2cm} | }
 \hline
& I & II & II & IV  \\ 
\hline                                    
\hline
$\nu \cdot \bar{\Omega}'(\beta_n)|_{\psi_J}$ &  $\nu x^{-2}$ & $\nu x^{-1}$ & $\nu x^{-1}$ & $\nu$  \\ 
\hline  
$x \cdot \bar{\Omega}'(\beta^+_n)|_{\psi_J}$&  $x^{-\frac{3}{2}}$ & $\nu x^{-1}$ & & $x^{\frac{1}{2}}$  \\ 
\hline  
$x^{-1} \cdot \bar{\Omega}'(\beta^-_n)|_{\psi_J}$&  $x^{-\frac{5}{2}}$ &  & $\nu x^{-1}$ & $-x^{-\frac{1}{2}}$  \\ 
\hline  
\end{tabular}}
\caption{\label{Skein2}Coefficients of the skein relation for the specialisation $\psi_J$}
\end{table}
Following the columns of this table, we see that the skein relation holds for each of the four cases I, II, III and IV. Based on this, we conclude that the link invariant $\bar{\Omega}'|_{\psi_J}(L)$ satisfies the skein relation:
\begin{equation}\label{skeinJ}
x \cdot \bar{\Omega}'|_{\psi_J}(L_{+})-x^{-1} \cdot \bar{\Omega}'|_{\psi_J}(L_{-})= 
\left( x^{\frac{1}{2}}-x^{-\frac{1}{2}} \right) \bar{\Omega}'|_{\psi_J}(L).
\end{equation}
Also, for the unknot we have: $$\bar{\Omega}'(\mathscr U)=1.$$
\begin{rmk}
If we start with the skein relation \eqref{skeinJ} and change the variable $x=q^{2}$ we obtain:
\begin{equation*}
q^{2} \cdot \bar{\Omega}'|_{\psi_J}(L_{+})-q^{-2} \cdot \bar{\Omega}'|_{\psi_J}(L_{-})= 
\left( q-q^{-1} \right) \bar{\Omega}'|_{\psi_J}(L).
\end{equation*}
Next, if we change the variable $q$ to $-q^{-1}$ we obtain the usual version of the skein relation for the Jones polynomial:
\begin{equation*}
q^{-2} \cdot \bar{\Omega}'|_{\psi_J}(L_{+})-q^{2} \cdot \bar{\Omega}'|_{\psi_J}(L_{-})= 
\left( q-q^{-1} \right) \bar{\Omega}'|_{\psi_J}(L).
\end{equation*}

\end{rmk}
This shows that the specialisation $\psi_J$ of the intersection form gives the normalised Jones polynomial:
\begin{equation}
\bar{\Omega}'|_{d=x^{-1}}(L,x)=\tilde{J}(L,x).
\end{equation}

\subsubsection{Alexander polynomial} In the next part we look at the specialisation $\psi_\Delta$, given by $d=-1$. We multiply the coefficients associated to $\beta_n$ by $-\nu$ and the ones associated to $\beta_n^{-}$ by $-1$ and we get:
\begin{table}[H]
\hspace{10mm}{\renewcommand{\arraystretch}{1.9}\begin{tabular}{| C{2.5cm} | C{2cm} | C{2cm} | C{2cm} | C{2cm} | }
 \hline
& I & II & II & IV  \\ 
\hline                                    
\hline
$(-\nu) \cdot \bar{\Omega}'(\beta_n)|_{\psi_{\Delta}}$ &  $-\nu$ & $\nu$ & $\nu$ & $-\nu$  \\ 
\hline  
$\bar{\Omega}'(\beta^+_n)|_{\psi_{\Delta}}$&  $-x^{\frac{1}{2}}$ & $\nu$ & & $x^{-\frac{1}{2}}$  \\ 
\hline  
-$\bar{\Omega}'(\beta^-_n)|_{\psi_{\Delta}}$&  $x^{-\frac{1}{2}}$ &  & $\nu$ & $-x^{\frac{1}{2}}$  \\ 
\hline  
\end{tabular}}
\caption{\label{Skein3}Coefficients of the skein relation for the specialisation $\psi_\Delta$}
\end{table}
This shows us that the associated skein relation is true in all four cases I, II, III, IV, so we have:
\begin{equation}\label{skeinA}
\bar{\Omega}'|_{\psi_{\Delta}}(L_{+})- \bar{\Omega}'|_{\psi_{\Delta}}(L_{-})= 
\left( x^{-\frac{1}{2}}-x^{\frac{1}{2}} \right) \bar{\Omega}'|_{\psi_{\Delta}}(L).
\end{equation}
\begin{rmk}
The skein relation \eqref{skeinA} with the change of variable $x$ to $-x$ gives the usual skein relation for the Alexander polynomial:
\begin{equation*}
\bar{\Omega}'|_{\psi_{\Delta}}(L_{-})- \bar{\Omega}'|_{\psi_{\Delta}}(L_{+})= 
\left( x^{\frac{1}{2}}-x^{-\frac{1}{2}} \right) \bar{\Omega}'|_{\psi_{\Delta}}(L).
\end{equation*}
\end{rmk}
Since for the unknot $\bar{\Omega}'(\mathscr U)=1$, we conclude that the specialisation $\psi_{\Delta}$ of the intersection form gives the Alexander polynomial:
\begin{equation}
\bar{\Omega}'|_{d=-1}(L)(x)=\Delta(L,x).
\end{equation}
This concludes the relations from Theorem \ref{Tsk}.
\subsection{Closed intersection model} For the closed intersection form, the computation of the relation that is satisfied by a crossing change follows with the same argument as the one for the open intersection form, and we obtain:
\begin{equation}\label{skeinJJ}
x \cdot \bar{\Omega}|_{\psi_J}(L_{+})-x^{-1} \cdot \bar{\Omega}|_{\psi_J}(L_{-})= 
\left( x^{\frac{1}{2}}-x^{-\frac{1}{2}} \right) \bar{\Omega}|_{\psi_J}(L)
\end{equation}
\begin{equation}\label{skeinAA}
\bar{\Omega}|_{\psi_{\Delta}}(L_{+})- \bar{\Omega}|_{\psi_{\Delta}}(L_{-})= 
\left( x^{-\frac{1}{2}}-x^{\frac{1}{2}} \right) \bar{\Omega}|_{\psi_{\Delta}}(L).
\end{equation}
The difference comes from the evaluation of this invariant on the unknot. 
More precisely, in section \ref{S:3} we computed the value of $\Omega$ on the trivial braid:
\begin{equation*}
\Omega(\mathbb I_1)=x^{\frac{1}{2}}(1+d) \in \Z[x^{\pm\frac{1}{2}},d^{\pm1}].
\end{equation*}
This means that the invariant $\bar{\Omega}$ evaluated on the unknot $\mathscr U$ is:
\begin{equation}\label{eq:v}
\bar{\Omega}(\mathscr U)=x^{\frac{1}{2}}(1+d) \in \Z[x^{\pm \frac{1}{2}}, d^{\pm 1}]/\left((d+1)(dx-1) \right).
\end{equation}
Thus, if we look at the specialisation $d=x^{-1}$ it has the normalisation:
\begin{equation*}
\bar{\Omega}|_{\psi_J}(\mathscr U)=x^{\frac{1}{2}}(1+x^{-1})= x^{\frac{1}{2}}+x^{-\frac{1}{2}} \in \Z[x^{\pm \frac{1}{2}}].
\end{equation*}
This relation together with the skein relation \ref{skeinJJ} concludes that the $\psi_J$ specialisation of the closed intersection form is the un-normalised Jones polynomial:  
\begin{equation}
\bar{\Omega}|_{\psi_J}(L)(x)=J(L,x). 
\end{equation}
On the other hand, following equation \eqref{eq:v} we remark that the $\psi_{\Delta}$ specialisation ($d=-1$) of the closed intersection form vanishes for the unknot:
\begin{equation*}
\bar{\Omega}|_{\psi_\Delta}(\mathscr U)=0.
\end{equation*}
On the other hand, it satisfies the skein relation from \eqref{skeinAA}. This shows that the specialisation $\psi_\Delta$ of this invariant vanishes for any link
\begin{equation}
\bar{\Omega}|_{\psi_\Delta}(L)=0
\end{equation}
and concludes the statement from Theorem \ref{Tsk'}.
\section{Formulas for these invariants as interpolations of Jones and Alexander polynomials} \label{S:5}
This section arose from joint discussions with Rinat Kashaev, and I would like to thank him  for this. In this part, we consider algebraic varieties which are quotients of the Laurent polynomial ring by the product of two irreducible factors without multiplicity. Then, if we have an invariant taking values in this algebraic variety, and we consider its specialisations associated to the two irreducible factors then this invariant in the variety is forced to be an interpolation between these two specialisations. Let us make it precise for our cases. 

So far we have the intersection form $\bar{\Omega}$ which we know that it is a link invariant. We want to describe the precise form of this invariant. We will see that the fact that $\bar{\Omega}(\beta_n)(x,d)$ recovers the Jones polynomial and vanishes through the second specialisation of coefficients forces it to be a multiple of the Jones polynomial.

On the other hand, something interesting happens with the open intersection form. In the second part of this section we will see that the fact that $\bar{\Omega}'(\beta_n)(x,d)$ is an element in the quotient ring which recovers the Jones and Alexander polynomials through the two specialisations forces it to be a specific interpolation between the Jones and Alexander polynomials. 

\subsection{The closed model in the quotient ring} Following Theorem \ref{Tsk} and Definition \ref{N} we have:
\begin{center}
\begin{tikzpicture}
[x=1.2mm,y=1.4mm]

% Nodes of the diagram
\node (b1)  at   (44,0)  {$0 \in \Z[x^{\pm 1}]$};
\node (b5)  at  (0,0) {$J(L)(x) \in \Z[x^{\pm \frac{1}{2}}]$};
\node (b4) at (22,15)   {$\bar{\Omega}(L)(x,d) \in \Z[x^{\pm \frac{1}{2}}, d^{\pm 1}]/\left((d+1)(dx-1) \right)$};
\node (s1)  at   (0,6)  {$d=x^{-1}$};
\node (s2)  at   (44,6)  {$d=-1$};
\draw[<-]  (b1)      to node [left,yshift=-2mm,font=\large]{$\psi_{\Delta}$}   (b4);
\draw[<-]   (b5) to node [right,yshift=-2mm,font=\large] {$\psi_{J}$}                        (b4);
\end{tikzpicture}
\end{center}
The specialisation $\psi_\Delta$ gives $\bar{\Omega}(L)(x,d)|_{d=-1}=0$, so there exists $B(x,d) \in \LL$ such that:
\begin{equation}
\bar{\Omega}(L)(x,d)=(d+1) \cdot B(x,d).
\end{equation}
On the other hand the specialisation $\psi_J$ gives $\bar{\Omega}(L)(x,d)|_{d=x^{-1}}=J(L)(x)$. Using this property combined with the above relation we have:
\begin{equation*}
J(L)(x)=(x^{-1}+1) \cdot B(x,d)|_{d=x^{-1}}.
\end{equation*}
In other words, in the ring $\Z[x^{\pm \frac{1}{2}}]$ we have:
\begin{equation*}
\frac{J(L)(x)}{(x^{-1}+1)}=B(x,d)|_{d=x^{-1}}.
\end{equation*}
This shows that there exists $B'(x,d) \in \LL$ with the property:
\begin{equation}
B(x,d)=\frac{J(L)(x)}{(x^{-1}+1)}+B'(x,d) \cdot (xd-1).
\end{equation}
This implies that that in the quotient ring we have:
\begin{equation*}
\bar{\Omega}(L)(x,d)=(d+1) \cdot \frac{J(L)(x)}{(x^{-1}+1)}+B'(x,d) \cdot (d+1)(xd-1)=(d+1) \cdot \frac{J(L)(x)}{(x^{-1}+1)}.
\end{equation*}

Using the normalised version of the Jones polynomial we obtain:
\begin{equation*}
\bar{\Omega}(L)(x,d)=x^{\frac{1}{2}}(d+1) \cdot \tilde{J}(L)(x).
\end{equation*}

This concludes the relation from Theorem \ref{THEOREM3}.

\subsection{The open model in the quotient ring}
In this part we study which information we get from the fact that the open intersection form recovers the Jones and Alexander polynomials of the closure. Let $L$ be a link and we choose a braid representative $\beta_n\in B_n$. Then, following Theorem \ref{Tsk} we have :
\begin{center}
\begin{tikzpicture}
[x=1.2mm,y=1.4mm]

% Nodes of the diagram
\node (b1)  at   (44,0)  {$\Delta(L)(x) \in \Z[x^{\pm 1}]$};
\node (b5)  at  (0,0) {$\tilde{J}(L)(x) \in \Z[x^{\pm \frac{1}{2}}]$};
\node (b4) at (22,15)   {$\bar{\Omega}'(\beta_n)(x,d) \in \Z[x^{\pm \frac{1}{2}}, d^{\pm 1}]/\left((d+1)(dx-1) \right)$};
\node (s1)  at   (0,6)  {$d=x^{-1}$};
\node (s2)  at   (44,6)  {$d=-1$};
\draw[<-]  (b1)      to node [left,yshift=-2mm,font=\large]{$\psi_{\Delta}$}   (b4);
\draw[<-]   (b5) to node [right,yshift=-2mm,font=\large] {$\psi_{J}$}                        (b4);
\end{tikzpicture}
\end{center}

We start with the specialisation $\psi_\Delta$ and we know:
\begin{equation}
\bar{\Omega}'(\beta_n)(x,d)|_{d=-1}=\Delta(L)(x).
\end{equation}
Then this means that the difference between the intersection form and the Alexander polynomial is in the kernel of the specialisation. More precisely, there exists $A(x,d)\in \LL$ such that:
\begin{equation*}
\bar{\Omega}'(\beta_n)(x,d)-\Delta(L)(x)=A(x,d) \cdot (d+1).
\end{equation*}
This is equivalent to:
\begin{equation}\label{rel1}
\bar{\Omega}'(\beta_n)(x,d)=\Delta(L)(x)+A(x,d) \cdot (d+1).
\end{equation}
Now, we look at the specialisation $\psi_J$ and we have $\bar{\Omega}'(\beta_n)(x,d)|_{d=x^{-1}}=\tilde{J}(L)(x)$, so:
\begin{equation*}
\tilde{J}(L)(x)=\Delta(L)(x)+A(x,d)|_{d=x^{-1}} \cdot (x^{-1}+1).
\end{equation*}
This shows that:
\begin{equation}
A(x,d)|_{d=x^{-1}}=\frac{\tilde{J}(L)(x)-\Delta(L)(x)}{(x^{-1}+1)}.
\end{equation}
This implies that there exists $A'(x,d)\in \LL$ such that:
\begin{equation*}
A(x,d)-\frac{\tilde{J}(L)(x)-\Delta(L)(x)}{(x^{-1}+1)}=A'(x,d) \cdot (xd-1).
\end{equation*}
The last relation gives:
\begin{equation}
A(x,d)=\frac{\tilde{J}(L)(x)-\Delta(L)(x)}{(x^{-1}+1)}+A'(x,d) \cdot (xd-1).
\end{equation}
Following the last relation and equation \eqref{rel1} we have:
\begin{equation*}
\bar{\Omega}'(\beta_n)(x,d)=\Delta(L)(x)+(d+1)\left(\frac{\tilde{J}(L)(x)-\Delta(L)(x)}{(x^{-1}+1)}+A'(x,d) \cdot (xd-1)\right).
\end{equation*}
This shows that the open intersection form in the quotient ring is the following interpolation:
\begin{equation}
\bar{\Omega}'(\beta_n)(x,d)=\Delta(L)(x)+(d+1)\frac{\tilde{J}(L)(x)-\Delta(L)(x)}{(x^{-1}+1)}.
\end{equation}
So $\bar{\Omega}'(\beta_n)$ is a well-defined link invariant, denoted by $\bar{\Omega}'(L)$ which has the formula:
\begin{equation}
\bar{\Omega}'(L)(x,d)=\Delta(L)(x)+(d+1)\frac{\tilde{J}(L)(x)-\Delta(L)(x)}{(x^{-1}+1)}.
\end{equation}
This concludes the statement of Theorem \ref{THEOREM2}.
\section{Example of computation}\label{S:6}
\subsection{Trefoil knot} Let us compute the intersection model for the trefoil knot $T$, seen as the closure of the braid $\sigma^3 \in B_2$. We have the following curves in the punctured disc $\mathscr D_5$:
 $$ (\sigma^3 \cup \mathbb I_{3})\mathscr S'\cap\mathscr T'$$
\vspace{-10mm}
\begin{center}
\begin{figure}[H]
\centering
\includegraphics[scale=0.6]{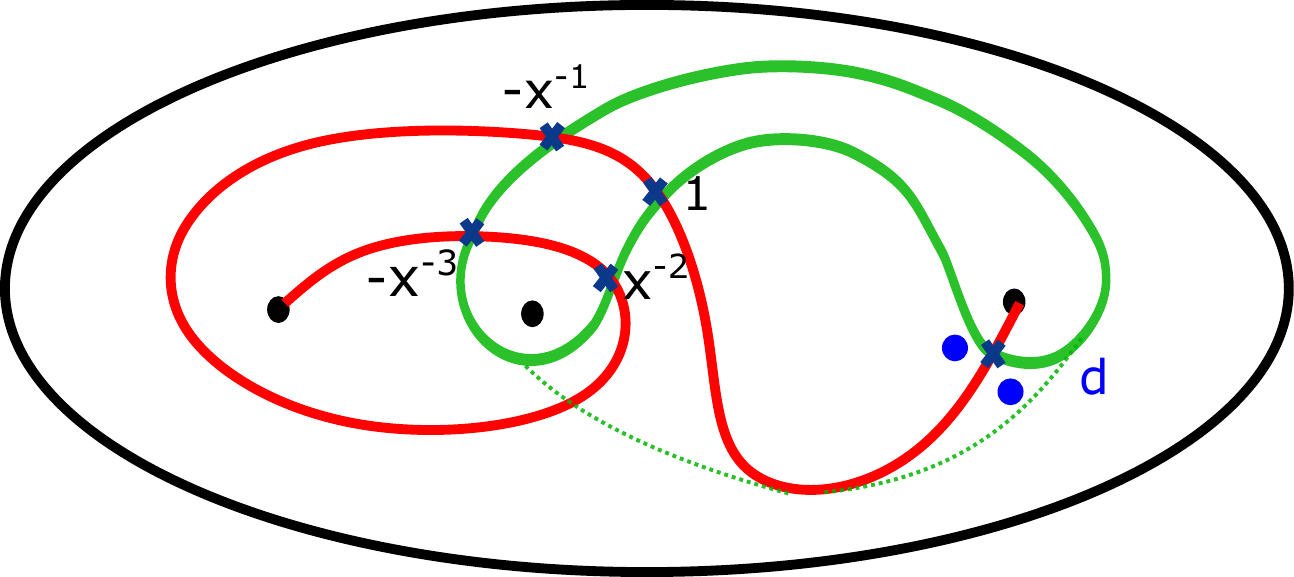}
\caption{Trefoil knot}
\end{figure}
\end{center}
\vspace{-5mm}
Then, we compute the gradings of the $5$ intersection points as in the above picture and we obtain:
\begin{equation}
\begin{aligned}
&\Omega'(\sigma^{3})(x,d)=x^2d^3
\left( -x^{-3}+x^{-2}-x^{-1}+1+d\right).
\end{aligned}
\end{equation}
In the next part, we compute the formula for this intersection in the quotient ring:
$$\bar{\Omega}'(T)(x,d) \in \Z[x^{\pm \frac{1}{2}}, d^{\pm 1}]/\left( (d+1)(dx-1)\right).$$
Replacing $d^3$ in terms of the basis using relation \eqref{relations} we obtain:
\begin{equation}
\begin{aligned}
&\bar{\Omega}'(T)(x,d)=x^2\left((x^{-2}-x^{-1}+1)d+x^2-x^{-1}\right)
\left( -x^{-3}+x^{-2}-x^{-1}+1+d\right).
\end{aligned}
\end{equation}
Now, we prove that this expression is an interpolation between the Jones and Alexander invariants of the trefoil knot, which have the following formulas:
\begin{equation}
\begin{aligned}
&\Delta(T,x)=x-1+x^{-1}\\
&J(T,q)=-x^{-4}+x^{-1}+x^{-3}.
\end{aligned}
\end{equation}
This means that we have:
\begin{equation}
\begin{aligned}
\bar{\Omega}'(T)(x,d)&=\left((\Delta(T)(x) \cdot xd-x+1\right)
\left( -x^{-3}+x^{-2}-x^{-1}+1+d\right)=\\
&=\left( (-x^{-2}+x^{-1}-1+x)d+xd^2\right)\Delta(T)(x)+\\
& \ \ \ +x^{-2}-x^{-1}+1-x-dx-x^{-3}+x^{-2}-x^{-1}+1+d.
\end{aligned}
\end{equation}
Replacing $d^2$ using formula \eqref{relations} we obtain: 
\begin{equation}
\begin{aligned}
\bar{\Omega}'(T)(x,d)&=\Delta(T)(x)+d(1-x^{-1})(1-x^{-1}+x^{-2}-x)+\\
&\ \ \ +(1-x^{-1})(1-x^{-1}+x^{-2}-x)=\\
&=\Delta(T)(x)+(d+1)(1-x^{-1})(1-x)(1+x^{-2}).
\end{aligned}
\end{equation}
On the other hand, the difference between Jones and Alexander polynomials of the trefoil has the expression:
\begin{equation*}
\tilde{J}(T)(x)-\Delta(T)(x)=(1+x^{-1})(1-x^{-1})(1-x)(1+x^{-2}).
\end{equation*}
From the previous two relations we conclude the interpolation model:
\begin{equation}
\bar{\Omega}'(T)(x,d)=\Delta(T)(x)+(d+1) \cdot \frac{\tilde{J}(T)(x)-\Delta(T)(x)}{(x^{-1}+1)}.
\end{equation}

\noindent {\itshape University of Geneva, Switzerland}

\noindent {\tt Cristina.Palmer-Anghel@unige.ch}

\noindent \href{http://www.cristinaanghel.ro/}{www.cristinaanghel.ro}


\begin{thebibliography}{99}
\bibitem {ADO} Y. Akustu, T. Deguchi, T. Ohtsuki -{ \em  Invariants of colored links}, J. Knot Theory Ramifications 1  161–184, (1992).
\bibitem{Cr} C. Anghel-{ \em $U_q(sl(2))$-quantum invariants from an intersection of two Lagrangians in a symmetric power of a surface}, arxiv.org/abs/2111.01125, 26 pages, (2021).
\bibitem{Cr2} C. Anghel-{ \em $U_q(sl(2))-$quantum invariants unified via intersections of embedded Lagrangians}, math.GT arxiv: 2010.05890, 26 pages, (2020).
\bibitem{Cr1} C. Anghel-{ \em Coloured Jones and Alexander polynomials as topological intersections of cycles in configuration spaces}, math.GT arXiv:2002.09390, 47 pages, (2020)
\bibitem{CrM} C. Anghel, M. Palmer- {\em Lawrence-Bigelow representations, bases and duality}, math.GT arXiv:2011.02388 (25 pages) (2020).
\bibitem{D} N. Dowlin- {\em A spectral sequence from Khovanov homology to knot Floer homology}, math.GT arXiv:1811.07848v1, 44 pages, (2018).
%\bibitem{KWZ} A. Kotelskiy, L. Watson, C. Zibrowius- {\em Immersed curves in Khovanov homology}, arXiv:1910.14584, (2019).
\bibitem{Big}  Stephen Bigelow - {A homological definition of the Jones polynomial.} In Invariants of knots and 3-manifolds (Kyoto, 2001), volume 4 of Geom. Topol. Monogr., pages 29-41. Geom. Topol. Publ., Coventry, (2002).
%\bibitem{Big0} Stephen Bigelow- {\em Braid groups are linear}, J. Amer. Math. Soc. 14, 471-486, (2001).
%\bibitem{Big2} Stephen Bigelow-{Homological representations of the Iwahori-Hecke algebra}, Geometry and Topology Monographs, Volume 7: Proceedings of the Casson Fest, Pages 493-507, (2004).
%\bibitem{Big4} S.Bigelow, V. Florens, A.Cattabriga {\em Alexander representation of tangles }, Acta Vietnamica Mathematica (2015) Vol. 40 (2) pp 339-352.
%\bibitem{Big3} Stephen Bigelow-{A homological definition of the HOMFLY polynomial}, Algebraic \& Geometric Topology 7, 1409-1440, (2007).
\bibitem{Fathi} F. B. Aribi- {\em Link invariants from L2-Burau maps of braids}, 24 pages, arXiv 2101.01678v4 (2021)
%\bibitem{Ito}  Tetsuya Ito - {Reading the dual Garside length of braids from homological and quantum representations.} Comm. Math. Phys., 335(1):345-367, (2015).
%\bibitem{GM} S. Gukov, C. Manolescu- {\em A two-variable series for knot complements} arXiv preprint arXiv:1904.06057, (2019).
%\bibitem{Ito2} Tetsuya Ito -{A homological representation formula of colored Alexander invariants} Adv. Math. 289, 142-160, (2016).  
%\bibitem{Ito3} Tetsuya Ito-{Topological formula of the loop expansion of the colored Jones polynomials}, Trans. Amer. Math. Soc., (2019) 
%\bibitem{JK} C. Jackson, T. Kerler- {\em The Lawrence-Krammer-Bigelow representations of the braid groups via $U_q(sl_2)$}, Adv. Math. 228, 1689-1717, (2011).
%\bibitem{J} V. Jones- {\em A Polynomial Invariant for Knots via von Neumann Algebras}, Bull. Amer. Math. Soc. (N.S.) 12, 103–111, (1985).
%\bibitem{K} M. Khovanov- {\em A categorification of the Jones polynomial}, Duke Math. J. 101, 359–426, (2000).
%\bibitem{KM} P. B. Kronheimer, T. S. Mrowka- {\em Khovanov homology is an unknot-detector}, Publ. Math. Inst. Hautes Etudes Sci. 113, 97–208, (2011).
%\bibitem{Martel} J. Martel -{\em A homological model for $U_q(sl(2)$ Verma-modules and their braid representations}, arXiv:2002.08785 (2020).
%\bibitem{M1} C. Manolescu- {\em Nilpotent slices, Hilbert schemes, and the Jones polynomial }, Duke Mathematical Journal 132, 311-369, (2006)
%\bibitem{O}T. Ohtsuki- {\em Quantum Invariants, A Study of Knots, 3-Manifolds and Their Sets}, Series on Knots and Everything, Volume 29, World Scientific, (2002).
%\bibitem{OZ} P. S. Ozsvath, Z. Szabo-{\em Holomorphic disks and topological invariants for closed three-manifolds}, Ann. of Math. (2), 159(3):1027–1158, (2004).
%\bibitem{Ras} J. Rasmussen- {\em Khovanov homology and the slice genus}, Invent. Math., 182(2):419–447, (2010).
\bibitem{RT} N. Reshetikhin, V. Turaev- {\em Invariants of 3-manifolds via link polynomials and quantum groups}, Invent. Math. 103, 547-597, (1991).
%\bibitem {SM} P. Seidel, I. Smith- {\em A link invariant from the symplectic geometry of nilpotent slices}, Duke Math. J.134, 453-514, (2006).

%\bibitem{Turaev}V. Turaev - {\em Quantum Invariants of Knots and 3-Manifolds}- Berlin, Boston: De Gruyter, (2016).

%\bibitem{Witt}E. Witten- {\em Quantum field theory and the Jones polynomial}, Comm. Math. Phys. 121, 351-399, (1989).

%\bibitem{K} R. Kashaev- {\em The hyperbolic volume of knots from the quantum dilogarithm}, Lett. Math. Phys. 39, 269-275, (1997).
%\bibitem {kass} C. Kassel, {\sl Quantum groups}, Springer (1995).
%\bibitem{kasstur} C. Kassel, V. Turaev, { \sl Braid groups}, Springer, (2008).
%\bibitem{Koh}  Toshitake Kohno - {Homological representations of braid groups and KZ connections.} J. Singul., 594-108, (2012).

%\bibitem{Koh2} Toshitake Kohno- {Quantum and homological representations of braid groups.} Configuration Spaces - Geometry, Combinatorics and Topology, Edizioni della Normale, 355-372, (2012). 
%\bibitem{Kr1} D. Krammer- {\em The braid group B4 is linear}, Invent. Math. 142, 451-486, (2000).
%\bibitem{Kr2} D. Krammer- {\em Braid groups are linear}, Ann. of Math. (2) 155, 131-156, (2002).
%\bibitem{Law1} R. J. Lawrence- {\em Homological representations of the Hecke algebra}, Comm. Math. Phys. 135, 141-19, (1990).
\bibitem{Law}  R. J. Lawrence - {A functorial approach to the one-variable Jones polynomial.} J. Differential Geom., 37(3):689-710, (1993).
%\bibitem{M1} C. Manolescu- {\em Nilpotent slices, Hilbert schemes, and the Jones polynomial }, Duke Mathematical Journal, Vol. 132, 311-369, (2006).
%\bibitem{MM} P. Melvin, H. Morton- {\em The coloured Jones function}, Comm. Math. Phys. 169, (1995), 501–520.
\bibitem{Ras} J. Rasmussen- {\em Some differentials on Khovanov-Rozansky homology}, Geom. Topol. 19, 3031-3104, (2015).
\bibitem {SM} P. Seidel, I. Smith- {\em A link invariant from the symplectic geometry of nilpotent slices}. Duke Math. J.134:453-514, (2006).

\end{thebibliography}
\end{document}